\documentclass[12pt]{amsart}
\usepackage{amssymb,amsmath}
\usepackage{tikz}
\usepackage{tikz-cd}
\usepackage{float}

\def\dint{\displaystyle\int}
\def\dsum{\displaystyle\sum}
\def\dprod{\displaystyle\prod}
\def\t{\mathbf{t}}
\def\x{\mathbf{x}}
\def\M{\bar{\mathcal{M}}}
\def\H{\mathcal{H}}
\def\<{\left<}
\def\>{\right>}
\def\S{\mathbf{\mathcal{S}}}

\newtheorem{thm}{Theorem}

\newtheorem{lemma}{Lemma}[section]
\newtheorem{lemma*}{Lemma}

\newtheorem{prop}{Proposition}[section]
\newtheorem{cor}{Corollary}[section]

\newtheorem{innercustomthm}{Theorem}
\newenvironment{customthm}[1]
  {\renewcommand\theinnercustomthm{#1}\innercustomthm}
  {\endinnercustomthm}

\begin{document}

\title{Two reconstruction theorems in permutation equivariant quantum K-theory}
\author{Dun Tang}

\begin{abstract}
In this paper, we first generalize the K-theoretic Ancestor-Descendant (AD) correspondence in \cite{perm7} to allow arbitrary permutative inputs.
With this version of AD correspondence, we reconstruct K-theoretical descendant $g=0$ invariants, and $g=1$ invariants with point target space, from $1$-point invariants of the corresponding genus.
In the appendix, we show that the graph of big $\mathcal{J}$ function also forms a Lagrangian cone in the permutation equivariant setting.
\end{abstract}

\maketitle

\tableofcontents

\newpage

\section*{Introduction}
\

The motivation of this paper is to reconstruct higher genus K-theoretical Gromov-Witten invariants.
The Quantum Hirzebruch-Riemann-Roch theorem \cite{perm9} for invariants of $g \geq 2$ involves in $permutation \ equivariant$ invariants rather than non-permutable ones \cite{nonperm}, even for non-permutable invariants of the point target space.
A useful tool for computing Gromov-Witten invariants is the Ancestor-Descendant correspondence.
In this paper, we develop the `completely permutable' Ancestor-Descendant correspondence (as a generalization of that in \cite{perm7}) and exploit some of its consequences in $g=0$ and $g=1$.

\subsection{Definitions of invariants}
\

We first recall the definition of correlators in permutation equivariant quantum K-theory.
Let $X$ be a compact Kahler manifold.
Let $\Lambda$ be a local $\lambda$-algebra (with Adams' operators $\Psi^k$) that contains Novikov's variables.
Let $K=K^0(X) \otimes \Lambda$.
The permutation equivariant K-theoretic Gromov-Witten invariants take in an input $\mathbf{t}_r \in \mathcal{K}_+=K[q^\pm]$ for each $r \in \mathbb{Z}_+$, and take values in $\Lambda$.

We define these invariants as follows \cite{perm9}.
The moduli space $\M_{g,n}(X,d)$ of degree $d$ stable maps from genus $g$ nodal curves with $n$ marked points to $X$ carries a virtual structural sheaf $\mathcal{O}^{vir}_{g,n,d}$ defined by Lee \cite{Lee01}.
The virtual structural sheaf $\mathcal{O}^{vir}_{g,n,d}$ is invariant under the $S_n$ action on $\M_{g,n}(X,d)$ given by renumbering marked points.
At the $i^{th}$ marked point there is the universal cotangent bundle $L_i$, and evaluation map $ev_i: \M_{g,n}(X,d) \to X$.
Given a permutation $h \in S_n$ with $\ell_r$ $r$-cycles, and inputs $\mathbf{t}_{r,k} (q) \in K[q^\pm]$ for each $r$-cycle $k = (k_1,\cdots,k_r)$, we set $T_{r,k} = \otimes_{i=1}^r(ev^\ast_{k_i}(\mathbf{t}_{r,k}))(L_{k_i})$.
Then each $T_{r,k}$ is a virtual $h$-orbi-bundle over $\M_{g,n}(X,d)$, with $h$ acting as $k$ permuting the components in the tensor product.
We define the correlators \cite{perm7}
\[\left<\mathbf{t}_{1,1}, \cdots, \mathbf{t}_{1,\ell_1}; \cdots , \mathbf{t}_{r,k}, \cdots \right>_{g,\ell,d} = (\prod_{r=1}^n r^{-\ell_r}) str_h H^\ast(\M_{g,n}(X,d); \mathcal{O}^{vir}_{g,n,d} \otimes (\otimes_{r,k} T_{r,k})).\]
Note that the right-hand side only depends on the type $\ell=(\ell_1,\cdots,\ell_n)$ of $h$, and thus the notation $\<\cdots\>_{g,\ell,d}$ make sense.

\subsection{Ancestor-Descendant correspondence}
\

The first theorem we prove in this paper is permutation equivariant K-theoretic Ancestor-Descendant correspondence, which generalizes that in \cite{perm7} and is the main tool used in this paper.
The main ingredients in the Ancestor-Descendant correspondence are the descendant potential $\mathcal{D}$ and the ancestor potential $\mathcal{A}$, living respectively in Fock spaces associated to certain symplectic spaces $\mathcal{K}^\infty, \bar{\mathcal{K}}^\infty$; and a symplectic linear map $\S:\mathcal{K}^\infty\to \bar{\mathcal{K}}^\infty$.
The definitions of $\mathcal{K}^\infty, \bar{\mathcal{K}}^\infty, \S$ are in Sections 1.4, 2.1.

We now construct $\mathcal{D, A}$.
Given a function $\mathcal{F}$ with inputs $\tau=(\tau_1,\tau_2,\cdots) \in K^\infty$, define $R_r$ by shifting inputs $\tau$'s (unless otherwise mentioned)
\[(R_r\mathcal{F})(\tau_1, \cdots, \tau_k,\cdots) = \mathcal{F}(\tau_r, \cdots, \tau_{kr},\cdots).\]
Let $v=\bar{v}=(1-q,\cdots,1-q,\cdots)$.

Given $\mathbf{t}=(\mathbf{t}_1, \cdots,\mathbf{t}_k, \cdots)$ with each $\t_k \in K[q^\pm]$, define generating functions 
\[\begin{array}{ll}
\mathcal{F}_g(\mathbf{t}) &= \dsum_{\ell,d}\frac{Q^d}{\ell!}\left<\mathbf{t}_1(L), \cdots, \mathbf{t}_1(L); \cdots , \mathbf{t}_r(L), \cdots \right>_{g,\ell,d},\\
\mathcal{D}(\mathbf{t}+v) &= \exp \dsum_{g \geq 0, k>0} \hbar^{k(g-1)}\frac{\Psi^k}k(R_k\mathcal{F}_g)(\t).
\end{array}\]
In the definition of $\mathcal{D}$, $R_k$'s shift inputs $\t_r$'s.

Given $\bar{\mathbf{t}}=(\bar{\mathbf{t}}_1, \cdots,\bar{\mathbf{t}}_k, \cdots)$ with each $\bar{\t}_k \in K[q^\pm]$, define generating functions 
\[\begin{array}{ll}
\bar{\mathcal{F}}_g(\bar{\mathbf{t}}) &= \dsum_{d,\ell,\ell_\tau} \frac{Q^d}{\ell!\ell_\tau!}\left<\bar{\mathbf{t}}_1(\bar{L}), \cdots, \bar{\mathbf{t}}_1(\bar{L}), \tau_1,\cdots,\tau_1;\cdots, \bar{\mathbf{t}}_r(\bar{L}), \cdots, \tau_r,\cdots \right>_{g,\ell+\bar{\ell},d} \\ 
\mathcal{A}(\bar{\mathbf{t}}+\bar{v}) &= \exp \dsum_{g \geq 0, k>0} \hbar^{k(g-1)}\frac{\Psi^k}k(R_k\bar{\mathcal{F}}_g)(\bar{\t}),
\end{array}\]
where $\bar{L}_i$ is the pull-back of the corresponding universal cotangent bundle along the forgetful map that forgets marked points with input $\tau$'s and the map to $X$.
Here in the definition of $\mathcal{A}$, $R_k$'s shift inputs $\bar{\t}_r$'s and $\tau_r$'s.

Given inputs $A_{1,1},\cdots, A_{1,\ell_1}, \cdots, A_{r,k}, \cdots \in K[L^{\pm}]$, with each $A_{r,k}$ as an input for an $r$-cycles, define
\[\begin{array}{ll}
&\left<\!\left<A_{1,1},\cdots, A_{1,\ell_1}; \cdots, A_{r,k}, \cdots \right>\!\right>_{g,\ell} \\
= &\dsum_{\bar{\ell},d} \frac{Q^d}{\bar{\ell}!} \cdot (\prod_r{r^{\ell_r}}) \left<A_{1,1}, \cdots, \tau_1,\cdots;\cdots, A_{r,k} \cdots, \tau_r, \cdots\right>_{g,\ell+\bar{\ell},d},\\
\end{array}\]
where $\ell,\bar{\ell}$ remembers the number of inputs $A$'s and $\tau$'s for each cycle length.

\begin{thm}
\[\mathcal{D}(\t)=\exp \left[\sum_k\frac{\Psi^k}k R_k (\left<\!\left< \ \right>\!\right>_{1,0} + \left<\!\left<\t_{2k}-\tau_{2k}+q-1\right>\!\right>_{0,1_2}/\hbar)\right] \cdot (\hat{\S}^{-1} \mathcal{A})(\t),\] where $R_k$'s shift $\tau$'s but not $\t_{2k}-\tau_{2k}+q-1$'s.
\end{thm}

Here $\hat{\S}^{-1}$ is the quantization of the Symplectic map $\S^{-1}$, as defined in \cite{Givental01}.

\subsection{Reconstruction theorem in $g=0$}
\

Let $\phi_\alpha$ be a basis of $K^0(X)$ and $\phi^\alpha$ its dual basis under the pairing $(\phi,\psi) = \chi(X, \phi \otimes \psi)$.
Recall that the Novikov's ring $\Lambda$ is a local $\lambda$-algebra with maximal ideal $\Lambda_+$.

\begin{thm}
The genus $0$ descendant potential $\mathcal{F}_0(\t)$ is recovered from $1$-point correlators 
\[\sum_\alpha \phi^\alpha \<\!\<\dfrac{\phi_\alpha}{1-qL}\>\!\>_{0,1_1}, \sum_\alpha \phi^\alpha \<\!\<\dfrac{\phi_\alpha}{1-qL}\>\!\>_{0,1_2}.\]
Moreover, there is a finite algorithm that computes $\mathcal{F}_0$ modulo $\Lambda_+^n$ for each $n$.
\end{thm}

\subsection{Reconstruction theorem of point target space in $g=1$}
\

Let $\bar{\x}_r=1-q+\bar{\t}_r$, $t_{r,0} = \t_r(1)$ and $\bar{t}_{2,2} = \frac12 \bar{\t}''_2(1)$.
Define
\[\begin{array}{ll}
\mathcal{F}_{1,2}^{perm} (x) = & \dfrac1{24} \<\!\<\dfrac1{1-x\bar{L}},1,1,1\>\!\>_{1,4_1} (\bar{\x}_1(x^{-1}))^4 \left(\dfrac{\partial \tau_2}{\partial t_{2,0}}\right)^2(1+\left(\dfrac{\partial \tau_2}{\partial t_{2,0}}\right)\bar{t}_{2,2}(-x^{-1}-1)); \\
\mathcal{F}_{1,3}^{perm} (x)  = & \dfrac16 \<\!\<\dfrac1{1-x\bar{L}}, 1,1\>\!\>_{1,3_1} \bar{\x}_1^3(x^{-1})\dfrac{\partial \tau_3}{\partial t_{3,0}}; \\
\mathcal{F}_{1,4}^{perm} (x) = & \dfrac14 \<\!\<\dfrac1{1-x\bar{L}}, 1, 1\>\!\>_{1,2_1+1_2} \bar{\x}^2_1(x^{-1}) \bar{\x}_2(x^{-2}) \dfrac{\partial \tau_4}{\partial t_{4,0}}; \\
\mathcal{F}_{1,6}^{perm} (x) = &\dfrac16 \<\!\<\dfrac1{1-x\bar{L}}, 1,1\>\!\>_{1,1_1+1_2+1_3} \bar{\x}_1 (x^{-1}) \bar{\x}_2(x^{-2}) \bar{\x}_3(x^{-3})\dfrac{\partial \tau_6}{\partial t_{6,0}}.
\end{array}\]
We remark that the $\ell$'s of the double-bracket correlators in $\mathcal{F}_{1, M}^{perm}$ remembers the number of ramification points (carrying $\x_r(x^{-1})$'s) of elliptic curves under the order $M$ permutation of marked points.

We pick $\tau$ so that $[\S(\t+1-q)]_+(1)=0$.
This $\tau$ is computed recursively in Section 3.3.

\begin{thm}
The genus $1$ generating function 
\[\mathcal{F}_1(\t) = F_1(\tau) + \frac1{24} \log\left(\frac{\partial \tau_1}{\partial t_{1,0}}\right) + \sum_{M=2,3,4,6} \sum_{a = 0,\infty}Res_a \mathcal{F}_{1,M}^{perm}(x) \frac{dx}x.\]
\end{thm}

\subsection{Lagragian cone formalism}
\

We also include an appendix on Givental's Lagrangian cone formalism in general permutation equivariant quantum K-theory as a standard consequence of (various structural equations and) the Ancestor-Descendant correspondence.
Let $\mathcal{L}$ be the range of the big $\mathcal{J}$-function 
\[\mathcal{J} (\t_1) = 1-q+\t_1 + \sum_{\ell, d,\alpha} \frac{Q^d}{\ell !} \phi^\alpha \<\frac{\phi_\alpha}{1-qL}, \t_1,\cdots,\t_1;\cdots,\t_r\cdots\>_{0,1_1+\ell,d},\]
viewed as a function in $\t_1$, with parameters $\t_2,\cdots,\t_r \cdots$.
Our main theorem in the appendix is

\begin{thm}
$\mathcal{L} = \cup_\tau (1-q) \S_{1,\tau}^{-1} \mathcal{K}_+$ is an overruled Lagrangian cone, i.e. its tangent spaces $T\subset {\mathcal K}$ are invariant under multiplication by $1-q$, and ${\mathcal L}$ is ruled by $(1-q)T$ at all points of which the tangent space is the same equal to $T$.
\end{thm}

\section*{Acknowledgments}
\

The author thanks Professor Alexander Givental for drawing his attention to $g=1$ quantum K-theory, Ancestor-Descendant correspondence, and various applications; for suggesting related problems, and for his patience and guidance throughout the process.

\section{Basic equations}
\

We generalize the following equations from \cite{perm7} and \cite{WDVV} to the `completely permutable' case.
Here $\M_{g,\ell}(X,d) = \M_{g,\sum_r r\ell_r}(X,d)$, and the subscript $\ell$ is for emphasizing the type of permutation on the marked points.

\subsection{String equation}
\

Set 
\[D\t(q)=\frac{\t(q)-\t(1)}{q-1}.\]

\begin{prop}
\[\begin{array}{ll}
&\left<1, \t_{1,1}, \cdots, \t_{1,\ell_1}; \cdots, \t_{r,k}, \cdots \right>_{0,1_1+\ell,d} \\
=& \left<\t_{1,1}, \cdots, \t_{1,\ell_1}; \cdots, \t_{r,k}, \cdots \right>_{0,\ell,d} + \dsum_{i=1}^{\ell_1} \left<\t_{1,1}, \cdots, D\t_{1,i} \cdots, \t_{1,\ell_1}; \cdots, \t_{r,k}, \cdots \right>_{0,\ell,d},
\end{array}\]
where $1_1$ stands for one $1$-cycle (and use $k_r$ for $k$ $r$-cycles).
\end{prop}

\begin{proof}
Let $ft: \M_{0,1_1+\ell}(X,d) \to \M_{0,\ell}(X,d)$ be the forgetful map that forgets the added marked point.
At the $i$ (out of $\ell$)-th marked point, define $\tilde{L}_i = ft^\ast (L_i)$.
Then we have
\[\left<1, \t_{1,1}(\tilde{L}), \cdots, \t_{1,\ell_1}(\tilde{L}); \cdots, \t_{r,k}(\tilde{L}), \cdots \right>_{0,1_1+\ell,d}=\left<\t_{1,1}(L), \cdots, \t_{1,\ell_1}(L); \cdots, \t_{r,k}(L), \cdots \right>_{0,\ell,d},\]
since $ft_\ast 1 = 1$ as fibers of $ft$ are of $g=0$.

Let $\sigma_i: \M_{0,\ell}(X,d) \to \M_{0,1_1+\ell}(X,d)$ be the section defined by the $i^{th}$ marked point.
\begin{itemize}
\item
The bundle $L_i$ is the co-normal bundle of the image of $\sigma_i$, so the push-forward $(\sigma_i)_\ast 1$ is $1-L_i$ on the image of $\sigma_i$, and thus $((\sigma_i)_\ast1)^r = (\sigma_i)_\ast (1-L_i)^{r-1}$.
\item
We have $(\sigma_i)_\ast 1=L_i - \tilde{L}_i$.
Indeed, this follows from the fact that $L_i$ and $\tilde{L}_i$ coincide everywhere outside $\sigma_i$, $(\sigma_i)_\ast 1 = 1-L_i$ on the image of $\sigma_i$, and $\sigma_i^\ast (\tilde{L}_i) = L_i, \sigma_i^\ast (L_i)=1$.
\item
With the fact that $L_i=1$ on the image of $\sigma_i$, we know that $(L_i-1)(\sigma_i)_\ast 1=0$.
\end{itemize}

Now, by Taylor's formula, we have 
\[\begin{array}{ll}
&\t(\tilde{L}_i)-\t(L_i) \\
=& \dsum_{k\geq0}\dfrac{\t^{(k)}(L_i)}{k!} ((-\sigma_i)_\ast 1)^k \\
=&(-\sigma_i)_\ast \left[\dsum_{k>0}\dfrac{\t^{(k)}(1)}{k!} (L_i-1)^{k-1} \right]\\
=&(-\sigma_i)_\ast \dfrac{\t(L)-\t(1)}{L-1}
\end{array}\]

Inserting $\t(\tilde{L}_i) = \t(L_i)-(\sigma_i)_\ast D\t(L_i)$ into the expression
\[\left<1, \t_{1,1}(\tilde{L}), \cdots, \t_{1,\ell_1}(\tilde{L}); \cdots, \t_{r,k}(\tilde{L}), \cdots \right>_{0,1_1+\ell,d},\]
we get a sum of correlators, with inputs $\t_{r,k}(L_i)$ or $-(\sigma_i)_\ast D\t_{r,k}(L_i)$ for the $i^{th}$ marked point that is in the $k^{th}$ $r$-cycle.
Note that the divisors $\sigma_i$ are disjoint, so the correlators vanish if two or more terms carry inputs $-(\sigma_i)_\ast D\t_{r,k}(L_i)$.
As a corollary, correlators also vanish if a point in an $r>1$ cycle carries $-(\sigma_i)_\ast D\t_{r,k}(L_i)$, as these points are permuted by the symmetry $h$ and do not contribute to the super-trace.
Thus, all but the following correlators vanish.
\begin{itemize}
\item
$\left<1,\t_{1,1}(L), \cdots, \t_{1,\ell_1}(L); \cdots, \t_{r,k}(L), \cdots \right>_{0,1_1+\ell,d}$;
\item
$\left<1,\t_{1,1}(L), \cdots, (\sigma_i)_\ast D\t_{1,i}(L), \cdots, \t_{1,\ell_1}(L); \cdots, \t_{r,k}(L), \cdots \right>_{0,1_1+\ell,d}$.
\end{itemize}
By the fact that $(\sigma_i)_\ast (L_j) = L_j$ if $j \neq i$, the second type of correlators simplify as
\[\left<\t_{1,1}, \cdots, D\t_{1,i} \cdots, \t_{1,\ell_1}; \cdots, \t_{r,k}, \cdots \right>_{0,\ell,d},\]
and string equation follows.
\end{proof}

\subsection{Dilaton equation}
\

\begin{prop}
\[\left<L-1, \t_{1,1}, \cdots, \t_{1,\ell_1}; \cdots, \t_{k,i}, \cdots \right>_{0,1_1+\ell,d} = (\ell_1-2)\left<\t_{1,1}, \cdots, \t_{1,\ell_1}; \cdots, \t_{k,i}, \cdots \right>_{0,\ell,d}.\]
\end{prop}

\begin{proof}
Identical to that of the string equation, we let $ft: \M_{0,1_1+\ell}(X,d) \to \M_{0,\ell}(X,d)$ be the forgetful map that forgets the added marked point, and $\tilde{L}_i = ft^\ast (L_i)$.
Then we have 
\[\begin{array}{ll}
&\left<L_1-1, \t_{1,1}(L), \cdots, \t_{1,\ell_1}(L); \cdots, \t_{k,i}(L), \cdots \right>_{0,1_1+\ell,d} \\
=& \left<L_1-1, \t_{1,1}(\tilde{L}), \cdots, \t_{1,\ell_1}(\tilde{L}); \cdots, \t_{k,i}(\tilde{L}), \cdots \right>_{0,1_1+\ell,d},
\end{array}\]
since $L_1-1$ vanish on the image of $\sigma_i$ for all $i\neq 1$.

To compute the push-forward of the right-hand side along $ft$, we need to compute $ft_\ast (L_1-1)$.
Over a point $(C,u) \in \M_{0,\ell}(X,d)$, the fiber of $ft$ is $C$, and thus the fiber of $ft_\ast (L_1-1)$ is $H^0(C, L_1-1) - H^1(C,L_1-1)$.
Given that the genus of $C$ is $0$ and that $\ell \neq 0$, we have $H^1(C,L_1-1)=0$.
On the other hand, $L_1$ differs from the dualizing sheaf of $C$ by the divisor of marked points, so $H^0(C, L_1)$ is the space of holomorphic differentials on $C$, with at most first order poles at marked points or nodes, such that the sum of residues at each node is $0$.
Indeed, when the genus of $C$ is $0$, these holomorphic differentials are uniquely determined by residues at marked points, while the only restriction of these residues is that they sum up to $0$.
The symmetry $h$ acts on these residues by permuting them.
Thus $H^0(C, L_1-1) - H^1(C,L_1-1)$ is the trivial bundle of dimension $n-2$, with $h$ acting on the $n$ dimensions as permuting the components, and on the $-2$ dimensions trivially.
So the super-trace of the $h$ action is $\ell_1-2$, and we have the dilaton equation.
\end{proof}

\subsection{WDVV equation}
\

Given inputs $A_1,\cdots,A_n \in K[q^{\pm}]$ as inputs for $1$-cycles, and $\tau=(\tau_1,\cdots) \in K^\infty$, recall that we defined
\[\left<\!\left<A_1,\cdots,A_n\right>\!\right>_{0,n} = \sum_{\ell,d} \frac{Q^d}{\ell!} \left<A_1,\cdots,A_n,\tau_1,\cdots;\cdots,\tau_r, \cdots\right>_{0,n_1+\ell,d}.\]

Pick a basis $\phi_\alpha$ of $K^0(X)$, and set
\[\begin{array}{ll}
(\varphi, \psi) &= \chi(X,\varphi \otimes \psi),\\
G_{\alpha\beta} &= (\phi_\alpha,\phi_\beta)+\left<\!\left< \phi_\alpha, \phi_\beta\right>\!\right>_{0,2}.
\end{array}\]
Let $[G^{\alpha\beta}]$ be the inverse of the matrix $[G_{\alpha\beta}]$.

\begin{lemma} \cite{WDVV}
The following expression is totally symmetric in $A_1, A_2, A_3, A_4$. 
\[\sum_{\mu,\nu} \left<\!\left<A_1,A_2,\phi_\mu\right>\!\right>_{0,3}G^{\mu\nu}\left<\!\left<\phi_\nu,A_3,A_4\right>\!\right>_{0,3}.\]
\end{lemma}

\begin{proof}
Consider the forgetful map $\lambda: \M_{0,4_1+\ell}(X,d) \to \bar{\mathcal{M}}_{0,4}=\mathbb{CP}^1$, where the identification $\bar{\mathcal{M}}_{0,4}=\mathbb{CP}^1$ is given by taking cross-ratio.
Note that $1 \in \mathbb{CP}^1$ corresponds to a degenerated curve in $\bar{\mathcal{M}}_{0,4}$, with the first two marked points on a nodal component, and the last two on another.
Similarly $0,\infty \in \mathbb{CP}^1$ also correspond to the other two nodal curves in $\bar{\mathcal{M}}_{0,4}$.
Let $A_1,\cdots, A_4$ be inputs possibly containing universal cotangent bundle $L_i$'s.
Use $\<\!\<A_1,A_2,A_3,A_4\>\!\>^z_{0,4}$ for the contribution from maps $\lambda^{-1}(z)$ to $\<\!\<A_1,A_2,A_3,A_4\>\!\>_{0,4}$.
Note that $\<\!\<A_1,A_2,A_3,A_4\>\!\>^z_{0,4}$ continuously depend on $z$, but takes value in the completely disconnected set $\Lambda$, so its value must be independent of $z$.

We first compute $\<\!\<A_1,A_2,A_3,A_4\>\!\>^1_{0,4}$.
By \cite{WDVV}, the virtual structural sheaf of $\lambda^{-1}(1)$ is identified with the structural sheaf of the following alternating sum of strata:

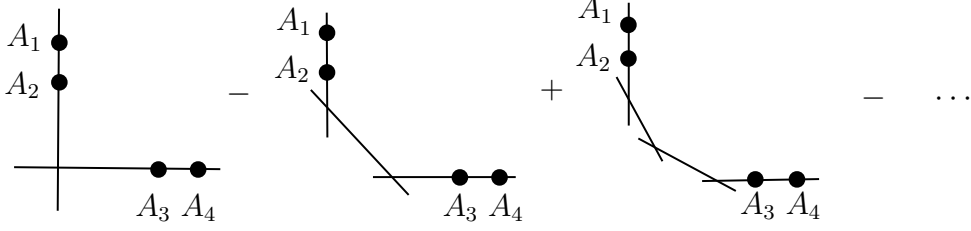
\tikzset{every picture/.style={line width=0.75pt}}
\begin{tikzpicture}[x=0.75pt,y=0.75pt,yscale=-1,xscale=1]
\draw    (70,19) -- (69.33,121) ;
\draw    (205,21) -- (205.33,84) ;
\draw    (357.33,16) -- (357.33,78) ;
\draw    (47.33,99) -- (151.33,100) ;
\draw    (228.33,104) -- (300.33,104) ;
\draw    (394.33,106) -- (453.33,105) ;
\draw    (197,60) -- (246.33,113) ;
\draw    (351.17,53.5) -- (374.33,96) ;
\draw    (362.17,84.5) -- (411.33,111) ;
\draw  [line width=3]  (68,36.17) .. controls (68,34.97) and (68.97,34) .. (70.17,34) .. controls (71.36,34) and (72.33,34.97) .. (72.33,36.17) .. controls (72.33,37.36) and (71.36,38.33) .. (70.17,38.33) .. controls (68.97,38.33) and (68,37.36) .. (68,36.17) -- cycle ;
\draw  [line width=3]  (68,56.17) .. controls (68,54.97) and (68.97,54) .. (70.17,54) .. controls (71.36,54) and (72.33,54.97) .. (72.33,56.17) .. controls (72.33,57.36) and (71.36,58.33) .. (70.17,58.33) .. controls (68.97,58.33) and (68,57.36) .. (68,56.17) -- cycle ;
\draw  [line width=3]  (118,100.17) .. controls (118,98.97) and (118.97,98) .. (120.17,98) .. controls (121.36,98) and (122.33,98.97) .. (122.33,100.17) .. controls (122.33,101.36) and (121.36,102.33) .. (120.17,102.33) .. controls (118.97,102.33) and (118,101.36) .. (118,100.17) -- cycle ;
\draw  [line width=3]  (138,100.17) .. controls (138,98.97) and (138.97,98) .. (140.17,98) .. controls (141.36,98) and (142.33,98.97) .. (142.33,100.17) .. controls (142.33,101.36) and (141.36,102.33) .. (140.17,102.33) .. controls (138.97,102.33) and (138,101.36) .. (138,100.17) -- cycle ;
\draw  [line width=3]  (203,31.17) .. controls (203,29.97) and (203.97,29) .. (205.17,29) .. controls (206.36,29) and (207.33,29.97) .. (207.33,31.17) .. controls (207.33,32.36) and (206.36,33.33) .. (205.17,33.33) .. controls (203.97,33.33) and (203,32.36) .. (203,31.17) -- cycle ;
\draw  [line width=3]  (203,51.17) .. controls (203,49.97) and (203.97,49) .. (205.17,49) .. controls (206.36,49) and (207.33,49.97) .. (207.33,51.17) .. controls (207.33,52.36) and (206.36,53.33) .. (205.17,53.33) .. controls (203.97,53.33) and (203,52.36) .. (203,51.17) -- cycle ;
\draw  [line width=3]  (270,104.17) .. controls (270,102.97) and (270.97,102) .. (272.17,102) .. controls (273.36,102) and (274.33,102.97) .. (274.33,104.17) .. controls (274.33,105.36) and (273.36,106.33) .. (272.17,106.33) .. controls (270.97,106.33) and (270,105.36) .. (270,104.17) -- cycle ;
\draw  [line width=3]  (290,104.17) .. controls (290,102.97) and (290.97,102) .. (292.17,102) .. controls (293.36,102) and (294.33,102.97) .. (294.33,104.17) .. controls (294.33,105.36) and (293.36,106.33) .. (292.17,106.33) .. controls (290.97,106.33) and (290,105.36) .. (290,104.17) -- cycle ;
\draw  [line width=3]  (355,27.17) .. controls (355,25.97) and (355.97,25) .. (357.17,25) .. controls (358.36,25) and (359.33,25.97) .. (359.33,27.17) .. controls (359.33,28.36) and (358.36,29.33) .. (357.17,29.33) .. controls (355.97,29.33) and (355,28.36) .. (355,27.17) -- cycle ;
\draw  [line width=3]  (355,44.17) .. controls (355,42.97) and (355.97,42) .. (357.17,42) .. controls (358.36,42) and (359.33,42.97) .. (359.33,44.17) .. controls (359.33,45.36) and (358.36,46.33) .. (357.17,46.33) .. controls (355.97,46.33) and (355,45.36) .. (355,44.17) -- cycle ;
\draw  [line width=3]  (419,105.17) .. controls (419,103.97) and (419.97,103) .. (421.17,103) .. controls (422.36,103) and (423.33,103.97) .. (423.33,105.17) .. controls (423.33,106.36) and (422.36,107.33) .. (421.17,107.33) .. controls (419.97,107.33) and (419,106.36) .. (419,105.17) -- cycle ;
\draw  [line width=3]  (440,105.17) .. controls (440,103.97) and (440.97,103) .. (442.17,103) .. controls (443.36,103) and (444.33,103.97) .. (444.33,105.17) .. controls (444.33,106.36) and (443.36,107.33) .. (442.17,107.33) .. controls (440.97,107.33) and (440,106.36) .. (440,105.17) -- cycle ;

\draw (42,25.4) node [anchor=north west][inner sep=0.75pt]    {$A_{1}$};
\draw (41,49.4) node [anchor=north west][inner sep=0.75pt]    {$A_{2}$};
\draw (107,111.4) node [anchor=north west][inner sep=0.75pt]    {$A_{3}$};
\draw (130,111.4) node [anchor=north west][inner sep=0.75pt]    {$A_{4}$};
\draw (178,17.4) node [anchor=north west][inner sep=0.75pt]    {$A_{1}$};
\draw (177,41.4) node [anchor=north west][inner sep=0.75pt]    {$A_{2}$};
\draw (330,11.4) node [anchor=north west][inner sep=0.75pt]    {$A_{1}$};
\draw (329,35.4) node [anchor=north west][inner sep=0.75pt]    {$A_{2}$};
\draw (262,111.4) node [anchor=north west][inner sep=0.75pt]    {$A_{3}$};
\draw (285,111.4) node [anchor=north west][inner sep=0.75pt]    {$A_{4}$};
\draw (412,110.4) node [anchor=north west][inner sep=0.75pt]    {$A_{3}$};
\draw (435,110.4) node [anchor=north west][inner sep=0.75pt]    {$A_{4}$};
\draw (153,55.4) node [anchor=north west][inner sep=0.75pt]    {$-$};
\draw (473,57.4) node [anchor=north west][inner sep=0.75pt]    {$-$};
\draw (311,53.4) node [anchor=north west][inner sep=0.75pt]    {$+$};
\draw (510,59.4) node [anchor=north west][inner sep=0.75pt]    {$\cdots $};
\end{tikzpicture}

\[\sum \M_{0,n_0+3}(X,d_0) \times_\Delta \M_{0,n_1+3}(X,d_1) - \sum \M_{0,n_0+3}(X,d_0) \times_\Delta \M_{0,n_1+2}(X,d_1) \times_\Delta \M_{0,n_2+3}(X,d_2) + \cdots ,\]
where $\Delta$ is the diagonal of $X \times X$ and the product is the fibered product over $\Delta$, given by evaluation maps at the added marked points.

As a result, the total contribution from these strata is 
\[\begin{array}{ll}&\dsum_{\alpha,\beta}\<\!\<A_1,A_2,\phi_\alpha\>\!\>_{0,3} g^{\alpha\beta} \<\!\<\phi_\beta,A_3,A_4\>\!\>_{0,3}\\
 - &\dsum_{\alpha,\beta,\gamma,\delta}\<\!\<A_1,A_2,\phi_\alpha\>\!\>_{0,3} g^{\alpha\beta} \<\!\<\phi_\beta,\phi_\gamma\>\!\>_{0,2}g^{\gamma\delta} \<\!\<\phi_\delta,A_3,A_4\>\!\>_{0,3} \\
+& \cdots,\end{array}\]
where $g^{\alpha\beta}$ are entries in the inverse of the matrix $[g_{\alpha\beta} = (\phi_\alpha,\phi_\beta)]$.
By using the matrix identity $[I+G]^{-1}=I-G+G^2- \cdots$, the contribution above is simplified as
\[\begin{array}{ll}&\left[\<\!\<A_1,A_2,\phi_\alpha\>\!\>_{0,3}\right] \left[g^{\alpha\beta}\right]\left[\delta_{\alpha\beta} + \dsum_\gamma g^{\alpha\gamma}\<\!\<\phi_\gamma,\phi_\beta\>\!\>_{0,2}\right]^{-1} \left[\<\!\<\phi_\beta,A_3,A_4\>\!\>_{0,3}\right]\\
=& \left[\<\!\<A_1,A_2,\phi_\alpha\>\!\>_{0,3}\right] \left[G^{\alpha\beta}\right] \left[\<\!\<\phi_\beta,A_3,A_4\>\!\>_{0,3}\right], \end{array}\]
where $\left[\cdots\right]$ stands for the matrix (row or column vectors) formed by $\cdots$, with rows and columns labeled by $\alpha,\beta$.
So 
\[\<\!\<A_1,A_2,A_3,A_4\>\!\>^1_{0,4} = \sum_{\mu,\nu} \left<\!\left<A_1,A_2,\phi_\mu\right>\!\right>_{0,3}G^{\mu\nu}\left<\!\left<\phi_\nu,A_3,A_4\right>\!\right>_{0,3}.\]

Next, $\<\!\<A_1,A_2,A_3,A_4\>\!\>^1_{0,4}$ is symmetric in $A_1,A_2$ and also in $A_3,A_4$, and that $\<\!\<A_1,A_2,A_3,A_4\>\!\>^1_{0,4} = \<\!\<A_1,A_2,A_3,A_4\>\!\>^0_{0,4}$, so $\<\!\<A_1,A_2,A_3,A_4\>\!\>^1_{0,4}$ is totally symmetric in $A_1,A_2,A_3,A_4$.
\end{proof}

We have the following direct corollary.
\begin{prop} 
The following expression is totally symmetric in $\alpha,\beta,\gamma,\delta$.
\[\sum_{\mu,\nu} \left<\!\left<\phi_\alpha,\phi_\beta,\phi_\mu\right>\!\right>_{0,3}G^{\mu\nu}\left<\!\left<\phi_\nu,\phi_\gamma,\phi_\delta\right>\!\right>_{0,3}.\]
\end{prop}

An alternative version of the WDVV equation also holds \cite{perm7}.
\begin{prop} 
For all $\varphi,\psi \in K$,
\[\begin{array}{ll}
&(\varphi,\psi)+(1-xy)\left<\!\left< \dfrac{\varphi}{1-xL}, \dfrac{\psi}{1-yL}\right>\!\right>_{0,2} \\
=& \dsum_{\alpha,\beta} \left((\varphi,\phi_\alpha)+\left<\!\left< \dfrac{\varphi}{1-xL}, \phi_\alpha\right>\!\right>_{0,2}\right)G^{\alpha\beta}\left((\phi_\beta,\psi)+\left<\!\left< \phi_\beta, \dfrac{\psi}{1-yL}\right>\!\right>_{0,2}\right).
\end{array}\]
\end{prop}

Here $\frac1{1-qL}$ is interpreted as follows.
There is a Laurent polynomial $P$ with zeros at roots of unities that annihilates $L$.
So $\Phi(q_1,q_2) = \frac{P(q_1)-P(q_2)}{q_1-q_2}$ is a Laurent polynomial, and we interpret $\frac1{1-qL}$ as the Laurent polynomial $\frac{\Phi(q^{-1},L)}{q \cdot P(q^{-1})}$ in $L$.

\begin{proof}
By Lemma 1.1, we have
\[\begin{array}{ll}
&\dsum_{\alpha,\beta} \left<\!\left<1, 1, \phi_\alpha\right>\!\right>_{0,3}G^{\alpha\beta}\left<\!\left< \phi_\beta, \dfrac{\psi}{1-xL},\dfrac{\varphi}{1-yL}\right>\!\right>_{0,3}\\
=&\dsum_{\alpha,\beta} \left<\!\left<1, \dfrac{\varphi}{1-xL}, \phi_\alpha\right>\!\right>_{0,3}G^{\alpha\beta}\left<\!\left< \phi_\beta, \dfrac{\psi}{1-yL},1\right>\!\right>_{0,3}.
\end{array}\]

By the string equation, we have
\[\begin{array}{ll}
\left<\!\left<1, \dfrac{\varphi}{1-qL}, \phi_\alpha\right>\!\right>_{0,3} & = \dfrac{(\varphi,\phi_\alpha)}{1-q} + (1+\dfrac{q}{1-q})\left<\!\left<\dfrac{\varphi}{1-qL}, \phi_\alpha\right>\!\right>_{0,2};\\
\left<\!\left<1, 1, \phi_\alpha\right>\!\right>_{0,3} & = (1,\phi_\alpha) + \<\!\<1,\phi_\alpha\>\!\>_{0,2};\\
\left<\!\left< 1, \dfrac{\psi}{1-xL},\dfrac{\varphi}{1-yL}\right>\!\right>_{0,3}&=\dfrac{(\psi,\varphi)}{(1-x)(1-y)} + (1+\dfrac{x}{1-x}+\dfrac{y}{1-y})\left<\!\left< \dfrac{\psi}{1-xL},\dfrac{\varphi}{1-yL}\right>\!\right>_{0,2}.
\end{array}\]

So the right hand of the equation is $\frac1{(1-x)(1-y)}$ that of our claim, while the left-hand side is
\[\left<\!\left< 1, \dfrac{\psi}{1-xL},\dfrac{\varphi}{1-yL}\right>\!\right>_{0,3},\]
which is also $\frac1{(1-x)(1-y)}$ that of our claim.
\end{proof}

\subsection{Consequences of WDVV equation}
\

By Corollary 3 in \cite{WDVV}, Proposition 1.3 together with the string equation implies that the metric $G_{\alpha\beta}$ is flat when viewed as a metric on the space of $\tau_1$ (i.e., with $\tau_2, \cdots$ fixed).

Define $\mathcal{K}$ as the space of $K$-valued rational functions with poles only at $0,\infty$, and roots of unities.
Put on $\mathcal{K}$ two Symplectic forms 
\[\begin{array}{ll}
\Omega(f,g) &= (Res_0+Res_\infty)(f(q),g(q^{-1}))\frac{dq}q,\\
\bar{\Omega}(f,g) &= (Res_0+Res_\infty) G(f(q),g(q^{-1}))\frac{dq}q,
\end{array}\] 
where $G$ is the inner product $\sum_{\alpha,\beta}G^{\alpha\beta} \phi_\alpha \otimes \phi_\beta$ on $K$, extended to $\mathcal{K}\subseteq K \otimes \mathbb{Q}((q))$ in a $\mathbb{Q}((q))$-linear way.
Then $\mathcal{K}_+ = K[q^{\pm}], \mathcal{K}_-=\{f:f(\infty)=0,f(0) \neq \infty\}$ is a Lagrangian polarization of both $(\mathcal{K},\Omega)$ and $(\mathcal{K},\bar{\Omega})$ \cite{perm9}.
Set 
\[S(q)\phi = \sum_{\alpha,\beta} \left((\phi,\phi_\alpha)+\left<\!\left< \frac{\phi}{1-L/q}, \phi_\alpha\right>\!\right>_{0,2}\right)G^{\alpha\beta}\phi_\beta.\]
Then $S(q):K \to \bar{\mathcal{K}}_-$ extends $\mathbb{Q}((q))$-linearly to a map $\mathcal{K} \to \bar{\mathcal{K}}$, which we also denote by $S(q)$.
Proposition 1.4 implies that $S(q)$ is Symplectic \cite{perm7}.

\section{Ancestor-Descendant correspondence}
\

\subsection{Statement of results}
\

Given a function $\mathcal{F}$ with inputs $\tau=(\tau_1,\tau_2,\cdots) \in K^\infty$, recall that $R_r$ is defined by
\[(R_r\mathcal{F})(\tau_1, \cdots, \tau_k,\cdots) = \mathcal{F}(\tau_r, \cdots, \tau_{kr},\cdots).\]
Define a metric $G_r=R_r(G)$ for each $r$.
Note that the correlators $\<\!\<\phi_\alpha,\phi_\beta\>\!\>_{0,2}$ in each entry $G_{\alpha\beta}$ in the matrix $G$ is a function in $\tau=(\tau_1,\cdots)$, and $R_r$ act on these functions.
By the flatness of $G$, each $G_r$ is a flat metric on the space of $\tau_r$.

Let $\mathcal{K}_r = \bar{\mathcal{K}}_r = \mathcal{K}$, with Symplectic forms $\Omega_r = \frac{\Psi^r}r \Omega, \bar{\Omega}_r = \frac{\Psi^r}r R_r\bar{\Omega}$, and Lagrangian polarizations $(\mathcal{K}_{r})_\pm = (\bar{\mathcal{K}}_{r})_\pm = \mathcal{K}_\pm \subseteq \mathcal{K}$.
Note that
\[\begin{array}{ll}
\Omega_r(f,g) &= (Res_0 + Res_\infty) \Psi^r(f(q^{-1}),g(q)) \frac{dq}q\\
\bar{\Omega}_r(f,g) &= (Res_0 + Res_\infty) \Psi^r(G_r(f(q^{-1}),g(q))) \frac{dq}q
\end{array}\]
Define Symplectic spaces and polarizations $(\mathcal{K}^\infty= \mathcal{K}^\infty_+ \oplus \mathcal{K}^\infty_-, \Omega^\infty)$ and $(\bar{\mathcal{K}}^\infty= \bar{\mathcal{K}}^\infty_+ \oplus \bar{\mathcal{K}}^\infty_-, \bar{\Omega}^\infty)$ as the direct sum of all of the components above.
Set dilaton vectors $v=\bar{v}=((1-q)1,\cdots,(1-q)1,\cdots)$.
Define a map $\S_r (q) = R_r(S(q)): \mathcal{K}_r \to \bar{\mathcal{K}}_r$ for each $r$, and $\S=\prod_{r=1}^\infty \S_r: \mathcal{K}^\infty \to \bar{\mathcal{K}}^\infty$.
Then $\S_r$'s are Symplectic, and so is $\S$.

Given $\mathbf{t}=(\mathbf{t}_1, \cdots,\mathbf{t}_k, \cdots) \in \mathcal{K}^\infty_+$ and $\bar{\mathbf{t}}=(\bar{\mathbf{t}}_1, \cdots) \in \bar{\mathcal{K}}^\infty_+$, recall that we defined generating functions 
\[\begin{array}{ll}
\mathcal{F}_g(\mathbf{t}) &= \dsum_{\ell,d}\frac{Q^d}{\ell!}\left<\mathbf{t}_1, \cdots, \mathbf{t}_1; \cdots , \mathbf{t}_r, \cdots \right>_{g,\ell,d},\\
\mathcal{D}(\mathbf{t}+v) &= \exp \dsum_{g \geq 0, k>0} \hbar^{k(g-1)}\frac{\Psi^k}k(R_k\mathcal{F}_g)(\t);\\
\bar{\mathcal{F}}_g(\bar{\mathbf{t}}) &= \dsum_{d,\ell,\bar{\ell}} \frac{Q^d}{\ell!\bar{\ell}!}\left<\bar{\mathbf{t}}_1(\bar{L}), \cdots, \bar{\mathbf{t}}_1(\bar{L}), \tau_1,\cdots,\tau_1;\cdots, \bar{\mathbf{t}}_r(\bar{L}), \cdots, \tau_r,\cdots \right>_{g,\ell+\bar{\ell},d} \\ 
\mathcal{A}_\tau(\bar{\mathbf{t}}+\bar{v}) &= \exp \dsum_{g \geq 0, k>0} \hbar^{k(g-1)}\frac{\Psi^k}k(R_k\bar{\mathcal{F}}_g)(\bar{\t}),
\end{array}\]
where $\bar{L}_i$ is the pull-back of the universal cotangent bundle at the corresponding marked point on $\bar{\mathcal{M}}_{g,\ell}$, along the forgetful map $\M_{g,\ell+\bar{\ell}}(X,d) \to \bar{\mathcal{M}}_{g,\ell}$.
Here in the definition of $\mathcal{A}$ and $\mathcal{D}$, $R_k$'s shift inputs $\t_r$, $\bar{\t}_r$ and $\tau_r$'s.

In this section, we prove the following theorem.

\begin{customthm}{1}
\[\mathcal{D}(\mathbf{x})=\exp \left[\sum_k\frac{\Psi^k}k R_k (\left<\!\left< \ \right>\!\right>_{1,0} + \left<\!\left<\mathbf{x}_{2k}-\tau_{2k}+q-1\right>\!\right>_{0,1_2}/\hbar)\right] (\hat{\S}^{-1} \mathcal{A})(\mathbf{x}),\] where $R_k$'s shift $\tau$'s but not $\mathbf{x}_{2k}-\tau_{2k}+q-1$'s.
\end{customthm}
We think of $\x$ as $\t+v+\tau$.

\subsection{Comparing $\bar{L}$ to $L$}
\

Recall that given inputs $A_{1,1},\cdots, A_{1,\ell_1}, \cdots, A_{r,k}, \cdots \in K[L^{\pm}]$ or $K[\bar{L}^{\pm}]$, with each $A_{r,k}$ as an input for an $r$-cycles
\[\begin{array}{ll}
&\left<\!\left<A_{1,1},\cdots, A_{1,\ell_1}; \cdots, A_{r,k}, \cdots \right>\!\right>_{g,\ell} \\
= &\dsum_{\bar{\ell},d} \frac{Q^d}{\bar{\ell}!} \cdot (\prod_r{r^{\ell_r}}) \left<A_{1,1}, \cdots, \tau_1,\cdots;\cdots, A_{r,k} \cdots, \tau_r, \cdots\right>_{g,\ell+\bar{\ell},d},\\
\end{array}\]
where $\ell,\bar{\ell}$ remembers the number of inputs $A$'s and $\tau$'s for each cycle length.

In this Subsection, we show that 
\begin{prop}
For stable $g,\ell$'s we have
\[\left<\!\left<\mathbf{t}_{1,1}(L), \cdots, \mathbf{t}_{1,\ell_1}(L); \cdots, \mathbf{t}_{r,k}(L), \cdots \right>\!\right>_{g,\ell} = 
\left<\!\left<\bar{\mathbf{t}}_{1,1}(\bar{L}), \cdots, \bar{\mathbf{t}}_{1,\ell_1}(\bar{L}); \cdots, \bar{\mathbf{t}}_{r,k}(\bar{L}), \cdots\right>\!\right>_{g,\ell},\]
where $\bar{\mathbf{t}}_{r,k}(q) = [\S_r(q)\mathbf{t}_{r,k}(q)]_+$ and $[\cdots]_+$ stands for taking positive part under the polarization $\bar{\mathcal{K}} = \bar{\mathcal{K}}_+ \oplus \bar{\mathcal{K}}_-$.
\end{prop}

For each $(C,u) \in \M_{g,n+\bar{n}}(X,d)$, there is a contraction map $ct_C: C \to \bar{C}$ that forgets the last $\bar{n}$ marked points and stabilizes the resulting curve.
Let $D_i$ be the virtual divisor of $\M_{g,n+\bar{n}}(X,d)$, consisting of points such that the $i^{th}$ marked point on the base curve $C$ lies on a component to be contracted under $ct_C$.
Note that $D_i$'s have normal intersections and self-intersections.
By Section 3 in \cite{WDVV}, we have

\begin{lemma}
The sheaf $\mathcal{O}(-D_i)$ is the push forward of virtual structural sheaves of the following components.
\end{lemma}

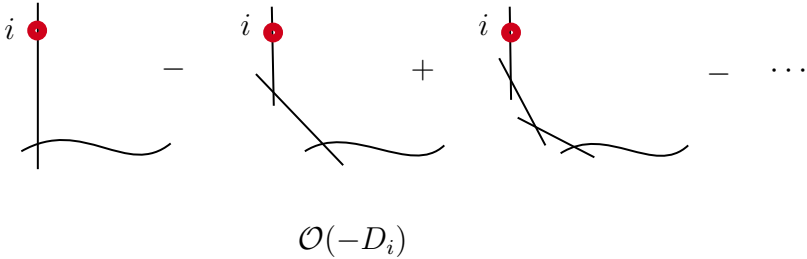
\tikzset{every picture/.style={line width=0.75pt}}
\begin{tikzpicture}[x=0.75pt,y=0.75pt,yscale=-1,xscale=1]
\draw    (76.33,48) -- (76.33,132) ;
\draw    (194.33,50) -- (195.33,100) ;
\draw    (186.33,84) -- (230.33,130) ;
\draw  [color={rgb, 255:red, 208; green, 2; blue, 27 }  ,draw opacity=1 ][line width=3]  (72.99,62.03) .. controls (72.99,60.37) and (74.33,59.03) .. (75.99,59.03) .. controls (77.65,59.03) and (78.99,60.37) .. (78.99,62.03) .. controls (78.99,63.68) and (77.65,65.03) .. (75.99,65.03) .. controls (74.33,65.03) and (72.99,63.68) .. (72.99,62.03) -- cycle ;
\draw  [color={rgb, 255:red, 208; green, 2; blue, 27 }  ,draw opacity=1 ][line width=3]  (191.99,63.03) .. controls (191.99,61.37) and (193.33,60.03) .. (194.99,60.03) .. controls (196.65,60.03) and (197.99,61.37) .. (197.99,63.03) .. controls (197.99,64.68) and (196.65,66.03) .. (194.99,66.03) .. controls (193.33,66.03) and (191.99,64.68) .. (191.99,63.03) -- cycle ;
\draw    (68,123) .. controls (99.33,105) and (121.33,137) .. (143.33,119) ;
\draw    (211,123) .. controls (234.33,108) and (264.33,136) .. (281.33,120) ;
\draw    (314.33,50) -- (315.01,97.02) ;
\draw    (309.01,76.02) -- (332.33,120) ;
\draw  [color={rgb, 255:red, 208; green, 2; blue, 27 }  ,draw opacity=1 ][line width=3]  (311.99,63.03) .. controls (311.99,61.37) and (313.33,60.03) .. (314.99,60.03) .. controls (316.65,60.03) and (317.99,61.37) .. (317.99,63.03) .. controls (317.99,64.68) and (316.65,66.03) .. (314.99,66.03) .. controls (313.33,66.03) and (311.99,64.68) .. (311.99,63.03) -- cycle ;
\draw    (340,124) .. controls (363.33,109) and (387.33,136) .. (404.33,120) ;
\draw    (318.33,106) -- (357.01,126.02) ;
\draw (57.99,54.43) node [anchor=north west][inner sep=0.75pt]  [rotate=-359.99]  {$i$};
\draw (444,77.4) node [anchor=north west][inner sep=0.75pt]    {$\cdots $};
\draw (206,159) node [anchor=north west][inner sep=0.75pt]   [align=left] {$\displaystyle \mathcal{O}( -D_{i})$};
\draw (176.99,52.43) node [anchor=north west][inner sep=0.75pt]  [rotate=-359.99]  {$i$};
\draw (137,74.4) node [anchor=north west][inner sep=0.75pt]    {$-$};
\draw (296.99,52.43) node [anchor=north west][inner sep=0.75pt]  [rotate=-359.99]  {$i$};
\draw (263,75.4) node [anchor=north west][inner sep=0.75pt]    {$+$};
\draw (412,77.4) node [anchor=north west][inner sep=0.75pt]    {$-$};

\end{tikzpicture}

In these pictures, line segments stand for rational components.
Define the depth of the $i^{th}$ marked point as the order of self-intersection, i.e., the total number of line segments in the picture.
In terms of a formula, these components are
\[-\dsum_{k\geq 1} (-1)^k \underbrace{X_{0,2} \times_{\Delta} \cdots \times_{\Delta} X_{0,2}}_{k} \times_{\Delta} X_{g,n},\]
where $X_{g,n}$ stands for the moduli space of stable maps from genus $g$ curves with $n$ distinguished marked points (with no restrictions put on the number of marked points that carry inputs $\tau$'s, and the degree), and $\Delta$ the diagonal in $X \times X$.

The virtually transverse section of $\bar{L}_i/L_i$ defined by the forgetful map has zero locus $D_i$, so $\bar{L}_i/L_i = \mathcal{O}(-D_i)$.
Apply Proposition 2.1 to all $i=1,\cdots,r$, and we have the following result.

\begin{cor}
The virtual sheaf $\mathcal{O}^{vir}_{\M_{g,n+\bar{n}}(X,d)} \cdot (1-\prod_{i=1}^r (\bar{L}_i/L_i))$ on $\M_{g,n+\bar{n}}(X,d)$ is the push forward of alternating sums of structural sheaves of components illustrated by the graph below, with signs $-(-1)^{\sum_i m_i}$, where $m_i$ is the depth of the $i^{th}$ marked point.
\end{cor}

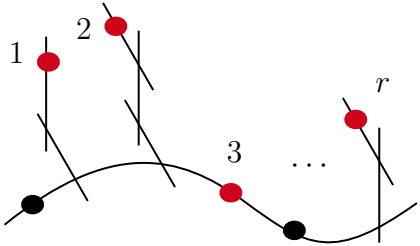
\tikzset{every picture/.style={line width=0.75pt}}

\begin{tikzpicture}[x=0.75pt,y=0.75pt,yscale=-1,xscale=1]
\draw    (14.33,123) .. controls (67.33,80) and (102.39,86.83) .. (134.86,110.92) .. controls (167.33,135) and (180.33,142) .. (222.33,112) ;
\draw  [color={rgb, 255:red, 0; green, 0; blue, 0 }  ,draw opacity=1 ][line width=3.75]  (25.24,112.83) .. controls (25.24,111.5) and (26.7,110.42) .. (28.49,110.42) .. controls (30.29,110.42) and (31.74,111.5) .. (31.74,112.83) .. controls (31.74,114.16) and (30.29,115.23) .. (28.49,115.23) .. controls (26.7,115.23) and (25.24,114.16) .. (25.24,112.83) -- cycle ;
\draw    (31,67) -- (56.33,111) ;
\draw    (185,59) -- (210.33,103) ;
\draw    (64,11) -- (89.33,55) ;
\draw    (75,60) -- (100.33,104) ;
\draw    (35.33,28) -- (35.33,86) ;
\draw    (82,25) -- (82,83) ;
\draw    (203.33,74) -- (203.33,132) ;
\draw  [color={rgb, 255:red, 208; green, 2; blue, 27 }  ,draw opacity=1 ][line width=3.75]  (188.24,69.83) .. controls (188.24,68.5) and (189.7,67.42) .. (191.49,67.42) .. controls (193.29,67.42) and (194.74,68.5) .. (194.74,69.83) .. controls (194.74,71.16) and (193.29,72.23) .. (191.49,72.23) .. controls (189.7,72.23) and (188.24,71.16) .. (188.24,69.83) -- cycle ;
\draw  [color={rgb, 255:red, 208; green, 2; blue, 27 }  ,draw opacity=1 ][line width=3.75]  (67.24,22.83) .. controls (67.24,21.5) and (68.7,20.42) .. (70.49,20.42) .. controls (72.29,20.42) and (73.74,21.5) .. (73.74,22.83) .. controls (73.74,24.16) and (72.29,25.23) .. (70.49,25.23) .. controls (68.7,25.23) and (67.24,24.16) .. (67.24,22.83) -- cycle ;
\draw  [color={rgb, 255:red, 0; green, 0; blue, 0 }  ,draw opacity=1 ][line width=3.75]  (157.24,125.83) .. controls (157.24,124.5) and (158.7,123.42) .. (160.49,123.42) .. controls (162.29,123.42) and (163.74,124.5) .. (163.74,125.83) .. controls (163.74,127.16) and (162.29,128.23) .. (160.49,128.23) .. controls (158.7,128.23) and (157.24,127.16) .. (157.24,125.83) -- cycle ;
\draw  [color={rgb, 255:red, 208; green, 2; blue, 27 }  ,draw opacity=1 ][line width=3.75]  (33.24,40.83) .. controls (33.24,39.5) and (34.7,38.42) .. (36.49,38.42) .. controls (38.29,38.42) and (39.74,39.5) .. (39.74,40.83) .. controls (39.74,42.16) and (38.29,43.23) .. (36.49,43.23) .. controls (34.7,43.23) and (33.24,42.16) .. (33.24,40.83) -- cycle ;
\draw  [color={rgb, 255:red, 208; green, 2; blue, 27 }  ,draw opacity=1 ][line width=3.75]  (125.24,106.83) .. controls (125.24,105.5) and (126.7,104.42) .. (128.49,104.42) .. controls (130.29,104.42) and (131.74,105.5) .. (131.74,106.83) .. controls (131.74,108.16) and (130.29,109.23) .. (128.49,109.23) .. controls (126.7,109.23) and (125.24,108.16) .. (125.24,106.83) -- cycle ;

\draw (15,29.4) node [anchor=north west][inner sep=0.75pt]    {$1$};
\draw (49,17.4) node [anchor=north west][inner sep=0.75pt]    {$2$};
\draw (125,79.4) node [anchor=north west][inner sep=0.75pt]    {$3$};
\draw (200,47.4) node [anchor=north west][inner sep=0.75pt]    {$r$};
\draw (157,88.4) node [anchor=north west][inner sep=0.75pt]    {$\cdots $};

\end{tikzpicture}

Colors when printed in black and white:

\begin{tikzpicture}[x=0.75pt,y=0.75pt,yscale=-1,xscale=1]
\draw  [line width=3.75]  (23,21.67) .. controls (23,20.19) and (24.19,19) .. (25.67,19) .. controls (27.14,19) and (28.33,20.19) .. (28.33,21.67) .. controls (28.33,23.14) and (27.14,24.33) .. (25.67,24.33) .. controls (24.19,24.33) and (23,23.14) .. (23,21.67) -- cycle ;
\draw  [color={rgb, 255:red, 208; green, 2; blue, 27 }  ,draw opacity=1 ][fill={rgb, 255:red, 255; green, 255; blue, 255 }  ,fill opacity=1 ][line width=3.75]  (113,21.33) .. controls (113,19.86) and (114.19,18.67) .. (115.67,18.67) .. controls (117.14,18.67) and (118.33,19.86) .. (118.33,21.33) .. controls (118.33,22.81) and (117.14,24) .. (115.67,24) .. controls (114.19,24) and (113,22.81) .. (113,21.33) -- cycle ;

\draw (43,12.33) node [anchor=north west][inner sep=0.75pt]   [align=left] {black};
\draw (133,12) node [anchor=north west][inner sep=0.75pt]   [align=left] {red};
\end{tikzpicture}

Next, we consider the $h$ action that cyclically permutes the red marked points labeled by $1,\cdots,r$.
All such $h$'s and the strata they fix are as follows, with the $r$ attached components identical and cyclically permuted by $h$.

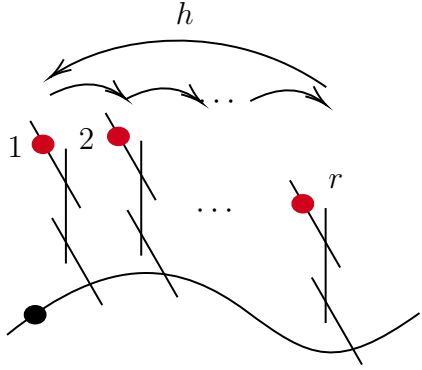
\tikzset{every picture/.style={line width=0.75pt}}
\begin{tikzpicture}[x=0.75pt,y=0.75pt,yscale=-1,xscale=1]
\draw    (14.33,176) .. controls (67.33,133) and (102.39,139.83) .. (134.86,163.92) .. controls (167.33,188) and (180.33,195) .. (222.33,165) ;
\draw  [color={rgb, 255:red, 0; green, 0; blue, 0 }  ,draw opacity=1 ][line width=3.75]  (25.24,165.83) .. controls (25.24,164.5) and (26.7,163.42) .. (28.49,163.42) .. controls (30.29,163.42) and (31.74,164.5) .. (31.74,165.83) .. controls (31.74,167.16) and (30.29,168.23) .. (28.49,168.23) .. controls (26.7,168.23) and (25.24,167.16) .. (25.24,165.83) -- cycle ;
\draw    (64,64) -- (89.33,108) ;
\draw    (75,113) -- (100.33,157) ;
\draw    (82,78) -- (82,136) ;
\draw  [color={rgb, 255:red, 208; green, 2; blue, 27 }  ,draw opacity=1 ][line width=3.75]  (67.24,75.83) .. controls (67.24,74.5) and (68.7,73.42) .. (70.49,73.42) .. controls (72.29,73.42) and (73.74,74.5) .. (73.74,75.83) .. controls (73.74,77.16) and (72.29,78.23) .. (70.49,78.23) .. controls (68.7,78.23) and (67.24,77.16) .. (67.24,75.83) -- cycle ;
\draw    (26,68) -- (51.33,112) ;
\draw    (37,117) -- (62.33,161) ;
\draw    (44,82) -- (44,140) ;
\draw  [color={rgb, 255:red, 208; green, 2; blue, 27 }  ,draw opacity=1 ][line width=3.75]  (29.24,79.83) .. controls (29.24,78.5) and (30.7,77.42) .. (32.49,77.42) .. controls (34.29,77.42) and (35.74,78.5) .. (35.74,79.83) .. controls (35.74,81.16) and (34.29,82.23) .. (32.49,82.23) .. controls (30.7,82.23) and (29.24,81.16) .. (29.24,79.83) -- cycle ;
\draw    (157,98) -- (182.33,142) ;
\draw    (168,147) -- (193.33,191) ;
\draw    (175,112) -- (175,170) ;
\draw  [color={rgb, 255:red, 208; green, 2; blue, 27 }  ,draw opacity=1 ][line width=3.75]  (160.24,109.83) .. controls (160.24,108.5) and (161.7,107.42) .. (163.49,107.42) .. controls (165.29,107.42) and (166.74,108.5) .. (166.74,109.83) .. controls (166.74,111.16) and (165.29,112.23) .. (163.49,112.23) .. controls (161.7,112.23) and (160.24,111.16) .. (160.24,109.83) -- cycle ;
\draw    (36,55) .. controls (50.64,47.36) and (62.25,47.93) .. (72.84,55.83) ;
\draw [shift={(74.33,57)}, rotate = 219.29] [color={rgb, 255:red, 0; green, 0; blue, 0 }  ][line width=0.75]    (10.93,-3.29) .. controls (6.95,-1.4) and (3.31,-0.3) .. (0,0) .. controls (3.31,0.3) and (6.95,1.4) .. (10.93,3.29)   ;
\draw    (74.33,57) .. controls (88.98,49.36) and (100.58,49.93) .. (111.18,57.83) ;
\draw [shift={(112.67,59)}, rotate = 219.29] [color={rgb, 255:red, 0; green, 0; blue, 0 }  ][line width=0.75]    (10.93,-3.29) .. controls (6.95,-1.4) and (3.31,-0.3) .. (0,0) .. controls (3.31,0.3) and (6.95,1.4) .. (10.93,3.29)   ;
\draw    (137,58) .. controls (151.64,50.36) and (163.25,50.93) .. (173.84,58.83) ;
\draw [shift={(175.33,60)}, rotate = 219.29] [color={rgb, 255:red, 0; green, 0; blue, 0 }  ][line width=0.75]    (10.93,-3.29) .. controls (6.95,-1.4) and (3.31,-0.3) .. (0,0) .. controls (3.31,0.3) and (6.95,1.4) .. (10.93,3.29)   ;
\draw    (175.33,51) .. controls (126.08,16.52) and (66.81,24.74) .. (37.33,45.06) ;
\draw [shift={(36,46)}, rotate = 324.09] [color={rgb, 255:red, 0; green, 0; blue, 0 }  ][line width=0.75]    (10.93,-3.29) .. controls (6.95,-1.4) and (3.31,-0.3) .. (0,0) .. controls (3.31,0.3) and (6.95,1.4) .. (10.93,3.29)   ;
\draw (13,74.4) node [anchor=north west][inner sep=0.75pt]    {$1$};
\draw (49,70.4) node [anchor=north west][inner sep=0.75pt]    {$2$};
\draw (175,93.4) node [anchor=north west][inner sep=0.75pt]    {$r$};
\draw (108,108.4) node [anchor=north west][inner sep=0.75pt]    {$\cdots $};
\draw (98,6.4) node [anchor=north west][inner sep=0.75pt]    {$h$};
\draw (109,53.4) node [anchor=north west][inner sep=0.75pt]    {$\cdots $};

\end{tikzpicture}

Following \cite{perm7}, we aim at eliminating $L$'s in the first (level $r$) input of correlators $\left<\phi L^a\bar{L}^b,\cdots\right>_{g,1_r+\ell+\bar{\ell},d}$.
The super-trace in the definitions of these correlators are
\[\begin{array}{ll}&str_h H^\ast (\M_{g,n+\bar{n}}(X,d), \mathcal{O}^{vir}_{\M_{g,n+\bar{n}}(X,d)} \otimes \dprod_{i=1}^r ev_i^\ast\phi L_i^a\bar{L}_i^b \otimes T) \\
=& str_h H^\ast (\M_{g,n+k}(X,d), \mathcal{O}^{vir}_{\M_{g,n+\bar{n}}(X,d)} \otimes \dprod_{i=1}^r ev_i^\ast\phi L_i^{a-1}\bar{L}_i^{b+1} \otimes T) \\
+& str_h H^\ast (\M_{g,n+k}(X,d), \mathcal{O}^{vir}_{\M_{g,n+\bar{n}}(X,d)} \otimes (\dprod_{i=1}^r ev_i^\ast\phi L_i^a\bar{L}_i^b)(1-\dprod_{i=1}^r \bar{L}_i/L_i) \otimes T),\end{array}\]
where $T$ is a certain virtual bundle on $\M_{g,n+k}(X,d)$ that keeps track of bundles pulled back from $X$ along evaluation maps.
Corollary 2.1 expresses the second term on the right-hand side as $\sum_Z  str_h H^\ast (Z, - \mathcal{O}_{Z}^{vir} \otimes (-1)^{mr} T|_Z)$, where $m$ is the depth of distinguished marked points, and $(-1)^{mr} = ((-1)^m)^{\otimes r}$ is permuted cyclically by $h$.
Thus in particular $str_h (-1)^{mr}=(-1)^m$ (not $(-1)^{mr}$).

\tikzset{every picture/.style={line width=0.75pt}}
\begin{tikzpicture}[x=0.75pt,y=0.75pt,yscale=-1,xscale=1]
\draw    (14.33,154) .. controls (67.33,111) and (102.39,117.83) .. (134.86,141.92) .. controls (167.33,166) and (180.33,173) .. (222.33,143) ;
\draw  [color={rgb, 255:red, 0; green, 0; blue, 0 }  ,draw opacity=1 ][line width=3.75]  (25.24,143.83) .. controls (25.24,142.5) and (26.7,141.42) .. (28.49,141.42) .. controls (30.29,141.42) and (31.74,142.5) .. (31.74,143.83) .. controls (31.74,145.16) and (30.29,146.23) .. (28.49,146.23) .. controls (26.7,146.23) and (25.24,145.16) .. (25.24,143.83) -- cycle ;
\draw    (64,42) -- (89.33,86) ;
\draw    (75,91) -- (100.33,135) ;
\draw    (82,56) -- (82,114) ;
\draw  [color={rgb, 255:red, 208; green, 2; blue, 27 }  ,draw opacity=1 ][line width=3.75]  (67.24,53.83) .. controls (67.24,52.5) and (68.7,51.42) .. (70.49,51.42) .. controls (72.29,51.42) and (73.74,52.5) .. (73.74,53.83) .. controls (73.74,55.16) and (72.29,56.23) .. (70.49,56.23) .. controls (68.7,56.23) and (67.24,55.16) .. (67.24,53.83) -- cycle ;
\draw    (26,46) -- (51.33,90) ;
\draw    (37,95) -- (62.33,139) ;
\draw    (44,60) -- (44,118) ;
\draw  [color={rgb, 255:red, 208; green, 2; blue, 27 }  ,draw opacity=1 ][line width=3.75]  (29.24,57.83) .. controls (29.24,56.5) and (30.7,55.42) .. (32.49,55.42) .. controls (34.29,55.42) and (35.74,56.5) .. (35.74,57.83) .. controls (35.74,59.16) and (34.29,60.23) .. (32.49,60.23) .. controls (30.7,60.23) and (29.24,59.16) .. (29.24,57.83) -- cycle ;
\draw    (157,76) -- (182.33,120) ;
\draw    (168,125) -- (193.33,169) ;
\draw    (175,90) -- (175,148) ;
\draw  [color={rgb, 255:red, 208; green, 2; blue, 27 }  ,draw opacity=1 ][line width=3.75]  (160.24,87.83) .. controls (160.24,86.5) and (161.7,85.42) .. (163.49,85.42) .. controls (165.29,85.42) and (166.74,86.5) .. (166.74,87.83) .. controls (166.74,89.16) and (165.29,90.23) .. (163.49,90.23) .. controls (161.7,90.23) and (160.24,89.16) .. (160.24,87.83) -- cycle ;
\draw  [color={rgb, 255:red, 74; green, 144; blue, 226 }  ,draw opacity=1 ][line width=3.75]  (52.24,129.83) .. controls (52.24,128.5) and (53.7,127.42) .. (55.49,127.42) .. controls (57.29,127.42) and (58.74,128.5) .. (58.74,129.83) .. controls (58.74,131.16) and (57.29,132.23) .. (55.49,132.23) .. controls (53.7,132.23) and (52.24,131.16) .. (52.24,129.83) -- cycle ;
\draw  [color={rgb, 255:red, 74; green, 144; blue, 226 }  ,draw opacity=1 ][line width=3.75]  (90.24,124.83) .. controls (90.24,123.5) and (91.7,122.42) .. (93.49,122.42) .. controls (95.29,122.42) and (96.74,123.5) .. (96.74,124.83) .. controls (96.74,126.16) and (95.29,127.23) .. (93.49,127.23) .. controls (91.7,127.23) and (90.24,126.16) .. (90.24,124.83) -- cycle ;
\draw  [color={rgb, 255:red, 74; green, 144; blue, 226 }  ,draw opacity=1 ][line width=3.75]  (185.24,159.83) .. controls (185.24,158.5) and (186.7,157.42) .. (188.49,157.42) .. controls (190.29,157.42) and (191.74,158.5) .. (191.74,159.83) .. controls (191.74,161.16) and (190.29,162.23) .. (188.49,162.23) .. controls (186.7,162.23) and (185.24,161.16) .. (185.24,159.83) -- cycle ;
\draw    (274.33,183) .. controls (327.33,140) and (362.39,146.83) .. (394.86,170.92) .. controls (427.33,195) and (440.33,202) .. (482.33,172) ;
\draw  [color={rgb, 255:red, 0; green, 0; blue, 0 }  ,draw opacity=1 ][line width=3.75]  (284.24,172.83) .. controls (284.24,171.5) and (285.7,170.42) .. (287.49,170.42) .. controls (289.29,170.42) and (290.74,171.5) .. (290.74,172.83) .. controls (290.74,174.16) and (289.29,175.23) .. (287.49,175.23) .. controls (285.7,175.23) and (284.24,174.16) .. (284.24,172.83) -- cycle ;
\draw    (321,4) -- (346.33,48) ;
\draw    (332,53) -- (357.33,97) ;
\draw    (339,18) -- (339,76) ;
\draw  [color={rgb, 255:red, 208; green, 2; blue, 27 }  ,draw opacity=1 ][line width=3.75]  (324.24,15.83) .. controls (324.24,14.5) and (325.7,13.42) .. (327.49,13.42) .. controls (329.29,13.42) and (330.74,14.5) .. (330.74,15.83) .. controls (330.74,17.16) and (329.29,18.23) .. (327.49,18.23) .. controls (325.7,18.23) and (324.24,17.16) .. (324.24,15.83) -- cycle ;
\draw    (283,8) -- (308.33,52) ;
\draw    (294,57) -- (319.33,101) ;
\draw    (301,22) -- (301,80) ;
\draw  [color={rgb, 255:red, 208; green, 2; blue, 27 }  ,draw opacity=1 ][line width=3.75]  (286.24,19.83) .. controls (286.24,18.5) and (287.7,17.42) .. (289.49,17.42) .. controls (291.29,17.42) and (292.74,18.5) .. (292.74,19.83) .. controls (292.74,21.16) and (291.29,22.23) .. (289.49,22.23) .. controls (287.7,22.23) and (286.24,21.16) .. (286.24,19.83) -- cycle ;
\draw    (414,38) -- (439.33,82) ;
\draw    (425,87) -- (450.33,131) ;
\draw    (432,52) -- (432,110) ;
\draw  [color={rgb, 255:red, 208; green, 2; blue, 27 }  ,draw opacity=1 ][line width=3.75]  (417.24,49.83) .. controls (417.24,48.5) and (418.7,47.42) .. (420.49,47.42) .. controls (422.29,47.42) and (423.74,48.5) .. (423.74,49.83) .. controls (423.74,51.16) and (422.29,52.23) .. (420.49,52.23) .. controls (418.7,52.23) and (417.24,51.16) .. (417.24,49.83) -- cycle ;
\draw  [color={rgb, 255:red, 74; green, 144; blue, 226 }  ,draw opacity=1 ][line width=3.75]  (309.24,91.83) .. controls (309.24,90.5) and (310.7,89.42) .. (312.49,89.42) .. controls (314.29,89.42) and (315.74,90.5) .. (315.74,91.83) .. controls (315.74,93.16) and (314.29,94.23) .. (312.49,94.23) .. controls (310.7,94.23) and (309.24,93.16) .. (309.24,91.83) -- cycle ;
\draw  [color={rgb, 255:red, 74; green, 144; blue, 226 }  ,draw opacity=1 ][line width=3.75]  (347.24,86.83) .. controls (347.24,85.5) and (348.7,84.42) .. (350.49,84.42) .. controls (352.29,84.42) and (353.74,85.5) .. (353.74,86.83) .. controls (353.74,88.16) and (352.29,89.23) .. (350.49,89.23) .. controls (348.7,89.23) and (347.24,88.16) .. (347.24,86.83) -- cycle ;
\draw  [color={rgb, 255:red, 74; green, 144; blue, 226 }  ,draw opacity=1 ][line width=3.75]  (442.24,121.83) .. controls (442.24,120.5) and (443.7,119.42) .. (445.49,119.42) .. controls (447.29,119.42) and (448.74,120.5) .. (448.74,121.83) .. controls (448.74,123.16) and (447.29,124.23) .. (445.49,124.23) .. controls (443.7,124.23) and (442.24,123.16) .. (442.24,121.83) -- cycle ;
\draw  [color={rgb, 255:red, 74; green, 144; blue, 226 }  ,draw opacity=1 ][line width=3.75]  (309.24,158.83) .. controls (309.24,157.5) and (310.7,156.42) .. (312.49,156.42) .. controls (314.29,156.42) and (315.74,157.5) .. (315.74,158.83) .. controls (315.74,160.16) and (314.29,161.23) .. (312.49,161.23) .. controls (310.7,161.23) and (309.24,160.16) .. (309.24,158.83) -- cycle ;
\draw  [color={rgb, 255:red, 74; green, 144; blue, 226 }  ,draw opacity=1 ][line width=3.75]  (347.24,152.83) .. controls (347.24,151.5) and (348.7,150.42) .. (350.49,150.42) .. controls (352.29,150.42) and (353.74,151.5) .. (353.74,152.83) .. controls (353.74,154.16) and (352.29,155.23) .. (350.49,155.23) .. controls (348.7,155.23) and (347.24,154.16) .. (347.24,152.83) -- cycle ;
\draw  [color={rgb, 255:red, 74; green, 144; blue, 226 }  ,draw opacity=1 ][line width=3.75]  (442.24,190.83) .. controls (442.24,189.5) and (443.7,188.42) .. (445.49,188.42) .. controls (447.29,188.42) and (448.74,189.5) .. (448.74,190.83) .. controls (448.74,192.16) and (447.29,193.23) .. (445.49,193.23) .. controls (443.7,193.23) and (442.24,192.16) .. (442.24,190.83) -- cycle ;

\draw (13,52.4) node [anchor=north west][inner sep=0.75pt]    {$1$};
\draw (49,48.4) node [anchor=north west][inner sep=0.75pt]    {$2$};
\draw (175,71.4) node [anchor=north west][inner sep=0.75pt]    {$r$};
\draw (108,86.4) node [anchor=north west][inner sep=0.75pt]    {$\cdots $};
\draw (242,100.4) node [anchor=north west][inner sep=0.75pt]    {$=$};
\draw (270,14.4) node [anchor=north west][inner sep=0.75pt]    {$1$};
\draw (306,10.4) node [anchor=north west][inner sep=0.75pt]    {$2$};
\draw (432,33.4) node [anchor=north west][inner sep=0.75pt]    {$r$};
\draw (365,48.4) node [anchor=north west][inner sep=0.75pt]    {$\cdots $};
\draw (370,113.4) node [anchor=north west][inner sep=0.75pt]    {$+$};
\draw (392,152.4) node [anchor=north west][inner sep=0.75pt]    {$\ddots $};
\draw (86,188.4) node [anchor=north west][inner sep=0.75pt]    {$Z$};
\draw (519,71.4) node [anchor=north west][inner sep=0.75pt]    {$Z_{2}$};
\draw (522,172.4) node [anchor=north west][inner sep=0.75pt]    {$Z_{1}$};

\end{tikzpicture}

Colors when printed in black and white:

\tikzset{every picture/.style={line width=0.75pt}}
\begin{tikzpicture}[x=0.75pt,y=0.75pt,yscale=-1,xscale=1]
\draw  [line width=3.75]  (23,21.67) .. controls (23,20.19) and (24.19,19) .. (25.67,19) .. controls (27.14,19) and (28.33,20.19) .. (28.33,21.67) .. controls (28.33,23.14) and (27.14,24.33) .. (25.67,24.33) .. controls (24.19,24.33) and (23,23.14) .. (23,21.67) -- cycle ;
\draw  [color={rgb, 255:red, 208; green, 2; blue, 27 }  ,draw opacity=1 ][fill={rgb, 255:red, 255; green, 255; blue, 255 }  ,fill opacity=1 ][line width=3.75]  (113,21.33) .. controls (113,19.86) and (114.19,18.67) .. (115.67,18.67) .. controls (117.14,18.67) and (118.33,19.86) .. (118.33,21.33) .. controls (118.33,22.81) and (117.14,24) .. (115.67,24) .. controls (114.19,24) and (113,22.81) .. (113,21.33) -- cycle ;
\draw  [color={rgb, 255:red, 74; green, 144; blue, 226 }  ,draw opacity=1 ][line width=3.75]  (193,20.33) .. controls (193,18.86) and (194.19,17.67) .. (195.67,17.67) .. controls (197.14,17.67) and (198.33,18.86) .. (198.33,20.33) .. controls (198.33,21.81) and (197.14,23) .. (195.67,23) .. controls (194.19,23) and (193,21.81) .. (193,20.33) -- cycle ;
\draw (43,12.33) node [anchor=north west][inner sep=0.75pt]   [align=left] {black};
\draw (133,12) node [anchor=north west][inner sep=0.75pt]   [align=left] {red};
\draw (209,13) node [anchor=north west][inner sep=0.75pt]   [align=left] {blue};
\end{tikzpicture}

We decompose $Z=Z_1 \times_{\Delta^r} Z_2$ as in the picture above.
Let $h_{1,r}, h_{2,r}$ be the cyclic permutations of the blue marked points according to the action of $h$ on the red marked points, and $h_1,h_2$ be the action of $h$ on other marked points in $Z_1, Z_2$.
Then the induced $h_1h_2$ action on $Z=Z_1 \times_{\Delta^r}Z_2$ is $h$.
Let $V_1,V_2$ be (virtual) $h_1,h_2$-bundles over $Z_1,Z_2$ and $V = V_1 \boxtimes V_2$.
Use $ev_{1,i}, ev_{2,i}$ for the evaluation maps at marked points permuted by $h_{1,r}, h_{2,r}$.

\begin{lemma}
\[str_{h} H^\ast (Z,V) = \sum_{\alpha,\beta} G_r^{\alpha\beta} str_{h_1h_{1,r}} H^\ast (Z_1,V_1 \cdot \prod_i ev_{1,i}^\ast \phi_\alpha) \cdot str_{h_2h_{2,r}} H^\ast (Z_2,V_2 \cdot \prod_i ev_{2,i}^\ast \phi_\beta).\]
\end{lemma}

\begin{proof}
There is a natural inclusion $Z=Z_1 \times_{\Delta^r} Z_2 \hookrightarrow Z_1 \times Z_2$, by viewing the $r$ nodes glued together as $2r$ marked points.
This inclusion intertwines the $h$ action on $Z$ and $h_1h_2h_{1,r}h_{2,r}$ action on $Z_1 \times Z_2$.
The K-theoretic virtual Euler class of the inclusion is $\prod_i \sum_{\alpha,\beta} G_r^{\alpha\beta} (ev_{1,i}^\ast \phi_\alpha \otimes ev_{2,i}^\ast \phi_\beta)$.
So 
\[str_{h} H^\ast (Z,V)=str_{h_1h_2h_{1,r}h_{2,r}} H^\ast(Z_1 \times Z_2, V \cdot\prod_i \sum_{\alpha,\beta} G_r^{\alpha\beta} (ev_{1,i}^\ast \phi_\alpha \otimes ev_{2,i}^\ast \phi_\beta)).\]
On the other hand, only terms $\sum_{\alpha,\beta} \prod_i G_r^{\alpha\beta} (ev_{1,i}^\ast \phi_\alpha \otimes ev_{2,i}^\ast \phi_\beta)$ in $\prod_i \sum_{\alpha,\beta} G_r^{\alpha\beta} (ev_{1,i}^\ast \phi_\alpha \otimes ev_{2,i}^\ast \phi_\beta)$ survive the supertrace, and the lemma follows.
\end{proof}

\begin{proof} (Proof of Proposition 2.1)
Use $\bar{Z}_2$ for one of the $r$ identical components in the product $Z_2$, $\bar{V}'_2$ for some bundle over $\bar{Z}_2$, and $V'_2 = \bar{V'}_2^{\boxtimes r}$. 
Then the super-trace
\[str_{h_2h_{2,r}}H^\ast(Z_2,V'_2) =str_{h_2h_{2,r}} H^\ast(\bar{Z}_2,\bar{V'}_2)^{\otimes r} =\Psi^r str_{(h_2h_{2,r})^r} H^\ast(\bar{Z}_2,\bar{V'}_2).\]
So the weighted sum of these super-traces, as the contributions from $Z_2$ over all possible $n,d$ and fixed $r$, with inputs $\phi^\alpha, \phi L^a$ for the first $r$ and last $r$ marked points, is $\Psi^rR_r(G^{\alpha\beta} \left<\!\left<\phi_\beta, \phi L^a\right>\!\right>_{0,2})$.
As a result,
\[\begin{array}{ll}
&\left<\!\left<\phi L^a\bar{L}^b,\cdots\right>\!\right>_{g,1_r+\ell}-\left<\!\left<\phi L^{a-1}\bar{L}^{b+1},\cdots\right>\!\right>_{g,1_r+\ell} \\
=& \sum_{\alpha,\beta} \left<\!\left<\cdots, \phi_\alpha \bar{L}^b\cdot G_r^{\alpha\beta} R_r(\left<\!\left<\phi_\beta, \phi L_i^a\right>\!\right>_{0,2})\right>\!\right>_{g,1_r+\ell}.
\end{array}\]
So
\[\left<\phi L^a \bar{L}^b\right|_r=\left<\phi L^{a-1} \bar{L}^{b+1}\right|_r+\left<\sum_{\alpha,\beta}\phi_\alpha \bar{L}^b G_r^{\alpha,\beta} \cdot R_r(\left<\!\left<\phi_\beta, \phi L^a\right>\!\right>_{0,2})\right|_r,\]
which implies that 
\[\left<\phi L^a\right|_r = \left<\phi \bar{L}^a\right|_r + \sum_{b=0}^{a-1} \sum_{\alpha,\beta}\left<\phi_\alpha, \bar{L}^b\right|_r G_r^{\alpha\beta} \cdot R_r (\left<\!\left<\phi_\beta, \phi L^{a-b}\right>\!\right>_{0,2}).\]
Similarly, for negative exponents, we have
\[\left<\phi L^{-a-1}\right|_r = \left<\phi \bar{L}^{-a-1}\right|_r - \sum_{b=0}^{a} \sum_{\alpha,\beta}\left<\phi_\alpha, \bar{L}^{-b-1}\right|_r G_r^{\alpha\beta} \cdot R_r (\left<\!\left<\phi_\beta, \phi L^{-a+b}\right>\!\right>_{0,2}).\]

This is concisely described as $\left<\mathbf{t}_r(L)\right|_r = \left<[\S_r(\bar{L})\mathbf{t}_r(\bar{L})]_+\right|_r$.
Indeed, by Cauchy's residue formula,
\[[f(q)]_+ = -(Res_{w=0}+Res_{w=\infty}) \frac{f(w)dw}{w-q}.\]
So for a Laurent polynomial $t(q)$, 
\[[\frac{t(q)}{1-L/q}]_+ = \frac{q\mathbf{t}(q)-L\mathbf{t}(L)}{q-L}.\]
Take $t(q) = \phi q^a$, we have
\[[\S_r(q)t(q)]_+ = \sum_{\alpha,\beta} ((\phi,\phi_\alpha)G^{\alpha\beta}_r \phi_\beta q^a + \sum_{b=0}^a R_r(\left<\!\left<\phi L^a, \phi_\alpha\right>\!\right>_{0,2}) G^{\alpha\beta}_r \phi_\beta q^{a-b}).\]
Apply
\[\sum_\alpha (\phi,\phi_\alpha) G^{\alpha\beta}_r = (\phi,\phi_\beta)-\sum_{\alpha} R_r(\left<\!\left<\phi, \phi_\alpha\right>\!\right>_{0,2}) G^{\alpha\beta}_r,\]
we have 
\[\left<\phi L^a\right|_r = \left<[\S_r(\bar{L})\phi \bar{L}^a)]_+\right|_r.\]
Similarly
\[\left<\phi L^{-a-1}\right|_r = \left<[\S_r(\bar{L})\phi \bar{L}^{-a-1})]_+\right|_r,\]
and so $\left<\mathbf{t}_{r,k}(L)\right|_r = \left<[\S_r(\bar{L})\mathbf{t}_{r,k}(\bar{L})]_+\right|_r = \left<\bar{\t}_{r,k}(\bar{L})\right|_r$ for $\bar{\t}_{r,k}= [\S_r\mathbf{t}_{r,k}]_+$.
Apply this identity to all inputs in $\left<\!\left<\mathbf{t}_{1,1}(L), \cdots, \mathbf{t}_{r,k}(L), \cdots \right>\!\right>_{g,\ell}$ proves Proposition 2.1 in the beginning of this subsection.
\end{proof}

\subsection{Proof of Theorem 1}
\

The sum of 
\[\left<\!\left<\mathbf{t}_1(L), \cdots, \mathbf{t}_1(L); \cdots, \mathbf{t}_r(L), \cdots \right>\!\right>_{g,\ell},\]
over all $\ell$'s (regardless of if $(g,\ell)$ is stable or not), is $\mathcal{F}_g(\tau+\t)$.
Set $\bar{\t} = [\S\t]_+$, i.e., $\bar{\t}_r = [\S_r\t_r]_+$ as in the last subsection, where $[\cdots]_+$ is the positive part of $\cdots$ under the Lagrangian polarization of ambient symplectic vector spaces.
By Proposition 2.1, the stable terms
\[\left<\!\left<\mathbf{t}_1(L), \cdots, \mathbf{t}_1(L); \cdots, \mathbf{t}_r(L), \cdots \right>\!\right>_{g,\ell}=\left<\!\left<\bar{\mathbf{t}}_1(\bar{L}), \cdots, \bar{\mathbf{t}}_1(\bar{L}); \cdots, \bar{\mathbf{t}}_r(\bar{L}), \cdots\right>\!\right>_{g,\ell}\]
sum up to $\bar{\mathcal{F}}_g(\bar{\mathbf{t}})$.
The unstable terms are $\<\!\< \ \>\!\>_{1,0}, \left<\!\left<\mathbf{t}_2\right>\!\right>_{0,1_2}, \<\!\< \ \>\!\>_{0,0}, \<\!\<\t_1\>\!\>_{0,1_1},\frac12\<\!\<\t_1,\t_1\>\!\>_{0,2_1}$, so we have
\[\begin{array}{ll}\mathcal{F}_g(\tau+\t) &= \bar{\mathcal{F}}_g(\bar{\mathbf{t}}) + \delta_{g,1} \<\!\< \ \>\!\>_{1,0}+ \delta_{g,0}\left<\!\left<\mathbf{t}_2\right>\!\right>_{0,1_2} \\
&+ \delta_{g,0} (\<\!\< \ \>\!\>_{0,0} + \<\!\<\t_1\>\!\>_{0,1_1}+\frac12\<\!\<\t_1,\t_1\>\!\>_{0,2_1}).\end{array}\]

To simplify this expression, we use the following facts.
\begin{lemma}
\[\begin{array}{ll}
\<\!\<1,A\>\!\>_{0,2_1} &= (A,\tau_1) + \<\!\<A\>\!\>_{0,1_1}, \\
\<\!\<L-1,A\>\!\>_{0,2_1}& = -\<\!\<A\>\!\>_{0,1_1}+\<\!\<A,\tau_1\>\!\>_{0,2_1},\\
\<\!\<L-1\>\!\>_{0,1_1} &= -2\<\!\< \ \>\!\>_{0,0}+\<\!\<\tau_1\>\!\>_{0,1_1}.
\end{array}\]
\end{lemma}
\begin{proof}
The first equation follows from the string equation.
\[\begin{array}{ll}
\<\!\<1,A\>\!\>_{0,2_1} &= \<1,A,\tau_1\>_{0,3_1,0}+\dsum_{\sum_r r\ell_r \geq 2 \ or \ d \neq 0} \dfrac{Q^d}{\ell!}\<1,A,\tau_1,\cdots\>_{0,2_1+\ell,d}\\
&= (A,\tau_1) + \dsum_{\sum_r r\ell_r \geq 2 \ or \ d \neq 0} \dfrac{Q^d}{\ell!}\<A,\tau_1,\cdots\>_{0,1_1+\ell,d} \\
&= (A,\tau_1) + \<\!\<A\>\!\>_{0,1_1}.\end{array}\]
Similarly, the second and third equations follow from the Dilaton equation.
\end{proof}

Let $\mathbf{x} = \t+v+\tau, \bar{\mathbf{x}} = \bar{\t}+\bar{v}$ be shifted inputs, where $v=\bar{v}=(1-q,1-q,\cdots)$ are the Dilaton vectors.

\begin{lemma}
\[[\S_1(q)(1-q+\tau_1)]_+=1-q.\]
\end{lemma}

\begin{proof}
Use
\[[f(q)]_+ = -(Res_{w=0}+Res_{w=\infty}) \frac{f(w)dw}{w-q},\]
and decompose rational functions in $w$ into partial fractions, we have
\[\left[\dfrac{1-q+\tau_1}{1-L/q}\right]_+= 1-q-L+\tau_1.\]
Together with Lemma 2.3, we have
\[(1-q+\tau_1,A)+ \<\!\<\left[\dfrac{1-q+\tau_1}{1-L/q}\right]_+,A\>\!\>_{0,2_1} = (1-q,A)+\<\!\<1-q,A\>\!\>_{0,2_1},\]
and so
\[[\S_1(q)(1-q+\tau_1)]_+=1-q.\]
\end{proof}

Lemma 2.4 together with $\bar{\t}_1=[\S_1\t_1]_+$ implies $\bar{\mathbf{x}}_1=[\S_1(\mathbf{x}_1)]_+$.
Apply $R_r$ and substitute $\bar{\mathbf{x}}_1, \mathbf{x}_1$ by $\bar{\mathbf{x}}_r, \mathbf{x}_r$, we have $\bar{\mathbf{x}}_r=[\S_r(\mathbf{x}_r)]_+$.
So $\bar{\mathbf{x}} = [\S\mathbf{x}]_+$.

\begin{lemma}
\[\<\!\< \ \>\!\>_{0,0} + \<\!\<\t_1\>\!\>_{0,1_1}+\frac12\<\!\<\t_1,\t_1\>\!\>_{0,2_1} = \frac12\<\!\<\mathbf{x}_1,\mathbf{x}_1\>\!\>_{0,2_1},\]
where $\mathbf{x}_1=\mathbf{t}_1-q+1+\tau_1$.
\end{lemma}

\begin{proof}
The right-hand side
\[\frac12\<\!\<\mathbf{t}_1-L+1+\tau_1,\mathbf{t}_1-L+1+\tau_1\>\!\>_{0,2_1}\]
splits into $3$ terms, namely the terms independent, linear, and quadratic in $\t_1$.
We compute these terms by Lemma 2.3. 
The term quadratic in $\t_1$ is 
\[\frac12\<\!\<\t_1,\t_1\>\!\>_{0,2_1}.\]
The term linear in $\mathbf{t}_1$ is
\[\begin{array}{ll}
&\<\!\<-L+1+\tau_1,\mathbf{t}_1\>\!\>_{0,2_1}\\
=&\<\!\<\mathbf{t}_1\>\!\>_{0,1_1}-\<\!\<\mathbf{t}_1,\tau_1\>\!\>_{0,2_1}+\<\!\<\tau_1,\mathbf{t}_1\>\!\>_{0,2_1}\\
=&\<\!\<\mathbf{t}_1\>\!\>_{0,1_1}.
\end{array}\]
We also have
\[\begin{array}{ll}\<\!\<L-1,L-1\>\!\>_{0,2_1} &= -\<\!\<L-1\>\!\>_{0,1_1} + \<\!\<L-1,\tau_1\>\!\>_{0,2_1}\\
&=2\<\!\< \ \>\!\>_{0,0}-2\<\!\<\tau_1\>\!\>_{0,1_1}+\<\!\<\tau_1,\tau_1\>\!\>_{0,2_1},\end{array}\]
so the term independent of $\mathbf{t}_1$ is
\[\begin{array}{ll}
&\dfrac12\<\!\<L-1,L-1\>\!\>_{0,2_1}-\<\!\<L-1,\tau_1\>\!\>_{0,2_1}+\dfrac12\<\!\<\tau_1,\tau_1\>\!\>_{0,2_1} \\
=& \<\!\< \ \>\!\>_{0,0}-\<\!\<\tau_1\>\!\>_{0,1_1}+\dfrac12\<\!\<\tau_1,\tau_1\>\!\>_{0,2_1}+\<\!\<\tau_1\>\!\>_{0,1_1}-\<\!\<\tau_1,\tau_1\>\!\>_{0,2_1} + \dfrac12\<\!\<\tau_1,\tau_1\>\!\>_{0,2_1} \\
=&\<\!\< \ \>\!\>_{0,0}.
\end{array}\]
\end{proof}
So 
\[\mathcal{F}_g(\mathbf{x}-v) = \bar{\mathcal{F}}_g(\bar{\mathbf{x}}-\bar{v}) + \delta_{g,1} \left<\!\left< \ \right>\!\right>_{1,0}+ \delta_{g,0}(\left<\!\left<\mathbf{t}_2\right>\!\right>_{0,1_2} + \dfrac{\<\!\<\mathbf{x}_1,\mathbf{x}_1\>\!\>_{0,2_1}}2),\]
and hence
\[\mathcal{D}(\mathbf{x})=\exp \left[\sum_k\frac{\Psi^k}k R_k \left(\left<\!\left< \ \right>\!\right>_{1,0} + (\left<\!\left<\mathbf{t}_{2k}\right>\!\right>_{0,1_2} + \frac12 \<\!\<\mathbf{x}_k,\mathbf{x}_k\>\!\>_{0,2_1})/\hbar\right)\right] \mathcal{A}(\bar{\mathbf{x}}),\]
where $R_k$ only shift $\tau$'s, and does not shift $\t$'s or $\mathbf{x}$'s.

On the other hand, by Proposition 5 in \cite{perm7}, we have
\[\hat{\S}^{-1}_1 \mathcal{A}(\mathbf{x}) = e^{\<\!\<\mathbf{x}_1,\mathbf{x}_1\>\!\>_{0,2_1}/2\hbar} \cdot \mathcal{A}([\S_1\mathbf{x}_1]_+,\mathbf{x}_2,\mathbf{x}_3,\cdots).\]
The Symplectic forms $\Omega_k,\bar{\Omega}_k$ on $\mathcal{K}_k,\bar{\mathcal{K}}_k$ are $\frac{\Psi^k}k R_k \Omega_1$ and $\frac{\Psi^k}k R_k \bar{\Omega}_1$, and that $\S_k=\Psi^k(\S_1)$,
so 
\[\hat{\S}^{-1}_k \mathcal{A}(\mathbf{x}) = e^{\frac{\Psi^k}k R_k (\<\!\<\mathbf{x}_k,\mathbf{x}_k\>\!\>_{0,2_1}/2\hbar)} \cdot \mathcal{A}(\mathbf{x}_1,\cdots,\mathbf{x}_{k-1},[\S_k\mathbf{x}_k]_+,\mathbf{x}_{k+1},\cdots).\]
Apply this formula to all inputs $\mathbf{x}_k$, we have
\[\hat{\S}^{-1} \mathcal{A}(\mathbf{x}) = e^{\sum_k\frac{\Psi^k}k R_k (\<\!\<\mathbf{x}_k,\mathbf{x}_k\>\!\>_{0,2_1}/2\hbar)} \cdot \mathcal{A}([\S\mathbf{x}]_+).\]

Together with $\bar{\mathbf{x}} = [S\mathbf{x}]_+$ and
\[\mathcal{D}(\mathbf{x})=\exp \left[\sum_k\frac{\Psi^k}k R_k \left(\left<\!\left< \ \right>\!\right>_{1,0} + \frac{\left<\!\left<\mathbf{t}_{2k}\right>\!\right>_{0,1_2} + \<\!\<\mathbf{x}_k,\mathbf{x}_k\>\!\>_{0,2_1}}{2\hbar}\right)\right] \mathcal{A}(\bar{\mathbf{x}}),\]
we have the Ancestor-Descendant correspondence
\[\mathcal{D}(\mathbf{x})=\exp \left[\sum_k\frac{\Psi^k}k R_k (\left<\!\left< \ \right>\!\right>_{1,0} + \left<\!\left<\mathbf{x}_{2k}-\tau_{2k}+q-1\right>\!\right>_{0,1_2}/\hbar)\right] (\hat{S}^{-1} \mathcal{A})(\mathbf{x}).\]
Here $\mathbf{x}_{2k}-\tau_{2k}+q-1$ is not shifted by $R_k$,

\section{Reconstruction of $g=0$ invariants}
\

Assume that $\Lambda$ is a local ring with maximal ideal $\Lambda_+$, that contain the Novikov variables, and that (coordinates under $\phi_\alpha$ of) $\tau_r, \t_r$'s are in $\Lambda_+^r$.
Assume also that $\cap_n \Lambda_+^n=0$.

In this section, we prove 

\begin{customthm}{2}
The genus $0$ descendant potential $\mathcal{F}_0(\t)$ is recovered from $1$-point correlators 
\[\sum_\alpha \phi^\alpha \<\!\<\dfrac{\phi_\alpha}{1-qL}\>\!\>_{0,1_1}, \sum_\alpha \phi^\alpha \<\!\<\dfrac{\phi_\alpha}{1-qL}\>\!\>_{0,1_2}.\]
Moreover, for each $n$, there is a finite algorithm that computes $\mathcal{F}_0$ modulo $\Lambda_+^{n+1}$.
\end{customthm}

\subsection{A formula for supertrace and graph sums}
\

We first recall a formula for the super-trace of sheaf cohomology from \cite{perm9}.
Let $h$ be a generator of a cyclic group $H$. 

The (virtual) orbifold $I^h\mathcal{M}$ is defined as follows.
In $I(\mathcal{M}/H)$, each stratum is characterized as the lift of some element $h^k \in H$ with respect to a stratum in $I\mathcal{M}$ labeled by $g$.
We label this stratum by $\tilde{h} = (g,h^k)$.
Then $I^h\mathcal{M}$ is the sub-orbifold of $I\mathcal{M}$ corresponding to the strata in $I(\mathcal{M}/H)$ associated specifically with $h=h^1$.

Locally, one can describe $I^h\mathcal{M}$ as follows.
Take a point $x \in \mathcal{M}^h$ and choose a local chart $U/G(x)$ centered at $x$.
The action of $h$ lifts in $|G(x)|$-ways to automorphisms of $U$.
Each such lift fixes some points in $U$.
Then, near $x$, $I^h\mathcal{M}$ is obtained by taking the union of the these fixed loci and quotienting by the action of $H$.

Let $N_{I^h\mathcal{M}/\mathcal{M}}$ be the normal bundle of inclusions of strata in $I^h\mathcal{M}$ to $\mathcal{M}$.
Let $V$ be a virtual $H$-orbibundle over a virtual $H$-orbifold $\mathcal{M}$.
The fake Euler characteristic is defined as $\chi^{fake}(M,V) = \int_{[M]} ch(V) \cdot td (M)$.

\begin{prop}
The super-trace $str_h H^\ast(\mathcal{M}, V)$ equals
\[\chi^{fake}\left(I^h\mathcal{M}, str_{\tilde{h}} \frac{V}{\wedge^\ast N^\ast_{I^h\mathcal{M}/\mathcal{M}}}\right).\]
\end{prop}

In particular, the super-trace $str_h H^\ast(\mathcal{M}, V)$ only depends on the restriction of $V$ to the fixed locus $\mathcal{M}^h$.

We remember various strata in $I^h\M_{g,n}$ (after forgetting all marked points with input $\tau$'s) by graphs decorated by the following data \cite{perm9}.
Each vertex carries a pair $(r,g)$, with $r \in \mathbb{Z}_+$ referred to as the level of the vertex and $g \in \mathbb{Z}_+$ the genus.
Each half-edge carries a pair $(r,\zeta)$, with $r \in \mathbb{Z}_+$ referred to as level, and $\zeta$ a root of unity of order $m(\zeta)$ referred to as the eigenvalue of the half-edge.
Edges connect half edges of the same level, with eigenvalues not inverse to each other.
The level of a half-edge is a multiple of the level of the vertex associated to it. 
The product $rm(\zeta)$ for each half-edge depends only on the vertex it is on.

In these decorated graphs, a type $(r,g)$ vertex represents a balanced multiplicity $r$ curve of total genus $g$.
A curve invariant under $h$ is balanced if it lives in the closure of smooth curves invariant under $h$.
A half edge of type $(r,\zeta)$ stands for a marked point with input $\t$, whose orbit under $h$ has cardinality $r$, and that $h^r$ acts on the cotangent space of the marked point with eigenvalue $\zeta$.
An edge stands for connecting two marked points along a node.
We allow arbitrarily many marked points with inputs $\tau$'s.

Fixed loci are written as the disjoint union of moduli spaces of maps represented by these decorated graphs, so the invariants are a sum of the contribution from these graphs.

\subsection{A vanishing result of ancestor invariants}
\

In this Subsection, we prove the following proposition.

\begin{prop}
If $\bar{\t}(1)=0$, then $\bar{\mathcal{F}}_0(\bar{\t})=0$.
\end{prop}

This proposition is a direct corollary of the following vanishing result.

\begin{lemma}
Assume that all $\bar{\t}_{r,k}$'s but at most $\bar{\t}_{1,1},\bar{\t}_{1,2}$ vanish at $1$.
Then the correlator 
\[\left<\!\left<\bar{\mathbf{t}}_{1,1}(\bar{L}), \cdots, \bar{\mathbf{t}}_{1,\ell_1}(\bar{L}); \cdots, \bar{\mathbf{t}}_{r,k}(\bar{L}), \cdots\right>\!\right>_{0,\ell}\]
vanishes.
\end{lemma}

\begin{proof}
For a graph to represent a non-empty family of $g=0$ curves, each vertex of level $r$ must satisfy one of the following conditions.
\begin{itemize}
\item
Two half-edges are of levels $r$ with eigenvalues roots of unities $\zeta, \zeta^{-1} \neq 1$.
All other half edges are of levels $rm(\zeta)$ and have eigenvalues $1$.
\item
All half edges are of level $r$ and have eigenvalues $1$.
\end{itemize}

Let $\Gamma$ be a decorated graph that labels the stratum $\mathcal{M} \subseteq I^h\M_{0,n}$.
Unless $\Gamma$ contains only one vertex $p$ (and no edges), let $p$ be the vertex of the highest level with only $1$ half-edge that forms edges.
In the latter case, since the graph is a tree, vertices with at most $1$ edges exist.
It follows from the choice of $p \in \Gamma$ and $g=0$ that $p$ must be connected to a half-edge along a half-edge of level $r(p)$.
Let $\mathcal{M}_p$ be the moduli space of curves defined by $p$, quotient by the $h$ action, which is a connected component of the space of stable maps to $B\mathbb{Z}_{mr}$.

Let $\mathcal{X} \subseteq I^h\M_{0,n+\bar{n}}(X,d)$ be the stratum mapped to $\mathcal{M}$ under the forgetful map $\M_{0,n+\bar{n}}(X,d) \to \M_{0,n}$.
The natural map $\mathcal{M} \to \mathcal{M}_p$ respects universal cotangent bundles at marked points in $\mathcal{M}_p$.
Therefore, at marked points in $\mathcal{X}$ that come from marked points in $\mathcal{M}_p$, the bundles $\bar{L}$ are the pull-back of bundles $L$ along the composition $\mathcal{X} \to \mathcal{M} \to \mathcal{M}_p$.
On the other hand, if $p$ has $k+2$ half edges, then $\mathcal{M}_p$ is of dimension $k-1$.
So by dimension arguments, the intersection $\chi^{fake}(\mathcal{M}_p, \prod_{i=1}^k (L^m_i-1) A) = 0$ for any $A$.
So by Proposition 3.1, the contribution to the correlators of the fixed loci $\mathcal{X}$ is $0$.
Summing over all $h, \mathcal{X}$ proves the lemma.
\end{proof}

\subsection{A recursive formula for $\tau$}
\

We define a map $T: K^\infty \to K^\infty$, by mapping $^1\tau = (^1\tau_1,\cdots,^1\tau_r, \cdots)$ to $^2\tau = (^2\tau_1,\cdots,^2\tau_r, \cdots)$, with
\[^2\tau_r = \t_r(1) + R_r\left[\dsum_{\alpha,\beta} \<\!\<LD\t_r(L), \phi_\alpha\>\!\>_{0,2_1}G_r^{\alpha\beta}\phi_\beta\right],\]
where the $\tau$ part in the double-bracket correlators is $^1\tau$, and $D\t(L)=\frac{\t(1)-\t(L)}{1-L}$.

We prove the following proposition in this section.

\begin{prop}
The sequence $\tau^{(n)} = T^n(0)$ converges to some $\tau \in K^\infty$ that makes $\bar{\t}(1)=0$.
\end{prop}

The map $T$ is a contraction mapping.

\begin{lemma}
If $^1\tau - ^2\tau \in \Lambda_+^n$, then $T(^1\tau)-T(^2\tau) \in \Lambda_+^{n+1}$.
\end{lemma}

\begin{proof}
The argument on a general sector is applying $R_r$ to the argument in the sector $\mathcal{K}_1$, so we only consider the sector $\mathcal{K}_1$.
From $^1\tau \equiv ^2\tau$ modulo $\Lambda_+^n$, we know that
\[\dsum_{\alpha,\beta,\gamma} \<\!\<\dfrac{\phi^\gamma}{1-qL}, \phi_\alpha\>\!\>_{0,2_1}\phi_\gamma \otimes G_1^{\alpha\beta}\phi_\beta\]
with inputs $^1\tau, ^2\tau$ for the $\tau$-part are congruent modulo $\Lambda_+^n$.
Now it follows from $\mathfrak{t}(q) = \t_r(q)-\t_r(1)+D\t_r(q) \in \Lambda_+$, and 
\[\dsum_{\alpha,\beta} \<\!\<\mathfrak{t}, \phi_\alpha\>\!\>_{0,2_1}G_1^{\alpha\beta}\phi_\beta = \Omega_1 \left(\dsum_{\alpha,\beta,\gamma} \<\!\<\dfrac{\phi^\gamma}{1-qL}, \phi_\alpha\>\!\>_{0,2_1}\phi_\gamma \otimes G_1^{\alpha\beta}\phi_\beta, \mathfrak{t}(q)\right)\]
that $T(^1\tau)-T(^2\tau) \in \Lambda_+^{n+1}$.
\end{proof}

Observe that $\tau^{(n)} = T^n(0)$ satisfy $\tau^{(0)}=0$ and $\tau^{(1)}=\t(1) \in \Lambda_+$.
So by Lemma 3.2 we know that $\tau^{(n-1)}-\tau^{(n)} \in \Lambda_+^n$, and $\tau^{(n)}$ converges to some fixed point $\tau \in K^{\infty}$ of $T$.

\begin{lemma}
$\bar{\mathbf{t}}(1)=0$ if and only if $\tau$ is a fixed point of $T$.
\end{lemma}

\begin{proof}
By using the fact
\[[f(q)]_+ = -(Res_{w=0}+Res_{w=\infty}) \frac{f(w)}{w-q}dw,\]
we get
\[[\frac{\mathbf{t}_1(q)+1-q}{1-L/q}]_+ = \frac{q(\mathbf{t}_1(q)+1-q)-L(\mathbf{t}_1(L)+1-L)}{q-L}.\]
The expression above is $(\mathbf{t}_1+D\mathbf{t}_1)(L) - L$ at $q=1$, so $\bar{\mathbf{t}}_1(1) = [\S_1(\mathbf{t}_1(q)+1-q)]_+(1)$ is
\[\sum_{\alpha,\beta} \phi_\alpha G^{\alpha\beta}((\phi_\beta,\mathbf{t}_1(1))+\<\!\<\phi_\beta,\mathbf{t}_1(L)+D\mathbf{t}_1(L)-L\>\!\>_{0,2_1}).\]
By Lemma 2.3, 
\[\<\!\<L,\phi_\alpha\>\!\>_{0,2_1}=(\tau_1,\phi_\alpha) + \<\!\<\tau_1,\phi_\alpha\>\!\>_{0,2_1},\]
and so
\[\begin{array}{ll}\bar{\mathbf{t}}_1(1) & =\dsum_{\alpha,\beta} \phi_\alpha G_1^{\alpha\beta}\left[(\phi_\beta,\mathbf{t}_1(1)-\tau_1) +\<\!\<\phi_\beta,\mathbf{t}_1(L)+D\mathbf{t}_1(L)-\tau_1\>\!\>_{0,2_1}\right]\\
& = \t_1(1) -\tau_1 + \dsum_{\alpha,\beta} \<\!\<LD\t_1(L), \phi_\alpha\>\!\>_{0,2_1}G_1^{\alpha\beta}\phi_\beta.
\end{array}\]
For general $\t_r$, the identity follows from a shift of inputs applied to the level $1$ identity above.
So $\bar{\t}(1) = T(\tau)-\tau$, and the proposition follows.
\end{proof}

\subsection{Proof of Theorem 2}
\

From (the proof of) Ancestor-Descendant correspondence, we have
\[\mathcal{F}_0(\mathbf{t}) = \bar{\mathcal{F}}_0(\bar{\mathbf{t}}) + (\left<\!\left<\mathbf{t}_2\right>\!\right>_{0,1_2} + \<\!\<\mathbf{t}_1(L) + 1-L,\mathbf{t}_1(L) +1-L\>\!\>_{0,2_1}/2),\]
where $\bar{\mathbf{t}} = [S(\t+1-q)]_++q-1$, that satisfy $\bar{\t}(1)=0$ by a smart choice of $\tau$.
By Proposition 3.2, $\bar{\mathcal{F}}_0(\bar{\mathbf{t}})=0$.

Let $\tau_1 = \sum_\alpha \tau_1^\alpha \phi_\alpha$.
By the WDVV equation (Proposition 1.4), we have 
\[\<\!\<\dfrac{\phi}{1-xL}, \dfrac{\psi}{1-yL}\>\!\>_{0,2_1} = \frac1{1-xy}(-(\phi,\psi)  + \sum_{\alpha,\beta}  A_\alpha G_1^{\alpha\beta}B_\beta),\]
where
\[\begin{array}{ll}
A_\alpha &= (\varphi,\phi_\alpha)+\dfrac{\partial}{\partial \tau_1^\alpha} \left<\!\left< \dfrac{\varphi}{1-xL}\right>\!\right>_{0,1_1},\\
B_\beta &= (\phi_\beta,\psi)+\dfrac{\partial}{\partial \tau_1^\beta} \left<\!\left<\dfrac{\psi}{1-yL}\right>\!\right>_{0,1_1},
\end{array}\]
are recovered from $1$-point correlators.

Also observe that
\[\begin{array}{ll}
\left<\!\left<\mathbf{t}_2\right>\!\right>_{0,1_2}&=\Omega_2\left(\dsum_\alpha \phi^\alpha \left<\!\left<\dfrac{\phi_\alpha}{1-qL}\right>\!\right>_{0,1_2}, \t_2(q)\right)\\
\<\!\<\mathbf{x}_1(L),\mathbf{x}_1(L)\>\!\>_{0,2_1}&=(\Omega_1 \otimes \Omega_1) \left(\dsum_{\alpha,\beta}\phi^\alpha \otimes \phi^\beta\<\!\<\dfrac{\phi_\alpha}{1-xL}, \dfrac{\phi_\beta}{1-yL}\>\!\>_{0,2_1}, \mathbf{x}_1(x)\otimes \mathbf{x}_1(y)\right),
\end{array}\]
where $\mathbf{x}_1(q) = \mathbf{t}_1(q) + 1-q$.
Therefore $\mathcal{F}_0(\mathbf{t})$ is reconstructed from 
\[\sum_\alpha \phi^\alpha \<\!\<\dfrac{\phi_\alpha}{1-qL}\>\!\>_{0,1_1}, \sum_\alpha \phi^\alpha \<\!\<\dfrac{\phi_\alpha}{1-qL}\>\!\>_{0,1_2}.\]

From Lemma 3.2, we know that to recover correlators modulo $\Lambda_+^{n+1}$, we only need the data of inputs modulo $\Lambda_+^n$, i.e., we only need to use $T^n(0)$ as inputs, so the algorithm is finite.

\section{$g=1$ invariants of point target space}
\

In this section, we reconstruct permutation equivariant quantum K-theoretic invariants in $g=1$ from general invariants in $g=0$ and 1-point invariants in $g=1$.
Theorem 1 implies
\[\mathcal{F}^1(\t) = F^1(\tau) + \bar{\mathcal{F}}^1(\bar{\t}),\]
where $(1-q+\bar{\t}) = [\mathcal{S}_\tau(1-q+\t)]_+$ and $\tau$ is chosen so that $\bar{\t}(1)=0$.
The inputs $\tau = (\tau_1,\tau_2, \cdots)$ and $\bar{\t} = (\bar{\t}_1,\bar{\t}_2,\cdots)$ are computed recursively in Section 3.3.

As before, let $\bar{\x}_r=1-q+\bar{\t}_r$.
Let $t_{r,0} = \t_r(1)$ be the constant term of $\t_r$, and $\bar{t}_{2,2} = \frac12 \bar{\t}''_2(1)$ be the coefficient of degree $2$ term of $\bar{\t}_2$ expanded at $q=1$.
Recall that 
\[\begin{array}{ll}
\mathcal{F}_{1,2}^{perm} (x) = & \dfrac1{24} \<\!\<\dfrac1{1-x\bar{L}},1,1,1\>\!\>_{1,4_1} (\bar{\x}_1(x^{-1}))^4 \left(\dfrac{\partial \tau_2}{\partial t_{2,0}}\right)^2(1+\left(\dfrac{\partial \tau_2}{\partial t_{2,0}}\right)\bar{t}_{2,2}(-x^{-1}-1)); \\
\mathcal{F}_{1,3}^{perm} (x)  = & \dfrac16 \<\!\<\dfrac1{1-x\bar{L}}, 1,1\>\!\>_{1,3_1} \bar{\x}_1^3(x^{-1})\dfrac{\partial \tau_3}{\partial t_{3,0}}; \\
\mathcal{F}_{1,4}^{perm} (x) = & \dfrac14 \<\!\<\dfrac1{1-x\bar{L}}, 1, 1\>\!\>_{1,2_1+1_2} \bar{\x}^2_1(x^{-1}) \bar{\x}_2(x^{-2}) \dfrac{\partial \tau_4}{\partial t_{4,0}}; \\
\mathcal{F}_{1,6}^{perm} (x) = &\dfrac16 \<\!\<\dfrac1{1-x\bar{L}}, 1,1\>\!\>_{1,1_1+1_2+1_3} \bar{\x}_1 (x^{-1}) \bar{\x}_2(x^{-2}) \bar{\x}_3(x^{-3})\dfrac{\partial \tau_6}{\partial t_{6,0}}.
\end{array}\]

\begin{customthm}{3}
The genus $1$ generating function 
\[\mathcal{F}_1(\t) = F_1(\tau) + \frac1{24} \log\left(\frac{\partial \tau_1}{\partial t_{1,0}}\right) + \sum_{M=2,3,4,6} \sum_{a = 0,\infty}Res_a \mathcal{F}_{1,M}^{perm}(x) \frac{dx}x.\]
\end{customthm}

Given a (virtual) sub-orbifold $\mathcal{M} \hookrightarrow \mathcal{N}$, we use $N_{\mathcal{M}/\mathcal{N}}$ for the (virtual) normal bundle of $\mathcal{M}$ in $\mathcal{N}$.

\subsection{Non-vanishing cases}
\

Recall that an $h$-invariant nodal curve is balanced if it is in the closure of smooth curves with $h$-action.

\begin{lemma}
The pre-image of unbalanced curves under the forgetful map $\M_{1,n+\bar{n}} \to \M_{1,n}$ does not contribute to $\bar{\mathcal{F}}_1(\bar{\t})$.
\end{lemma}

\begin{proof}
Recall that we remember these strata by decorated graphs (\cite{perm9}, see also Section 3.1), whose vertices stand for balanced curves and edges for connecting two curves along an unbalanced node.
By Lemma 3.1, a graph does not contribute to the Ancestor invariant if it contains a $g=0$ vertex
\begin{enumerate}
\item
With no edges emerging from unramified special points;
\item
Is unramified and has at most two edges.
\end{enumerate}

Now let $\Gamma$ be the minimum part of the graph with genus $1$.
By Lemma 3.1, no edges, and hence no other vertices, can exist outside $\Gamma$.
Next, $\Gamma$ is either a $g=1$ vertex or a cycle of $g=0$ vertices.
In the latter case, each vertex on this cycle has two edges connected along ramification points (if ramified), thus does not contribute to the Ancestor invariants.
Remark that two $\mathbb{CP}^1/\mathbb{Z}_2$ connected along a pair of unramified points is illegal, as that node is balanced.
\end{proof}

All non-vanishing cases are as follows.

\begin{table}[H]
\centering       
\begin{tabular}{c | c | c | c}                       
Case & Order of $h$ & $\M$ & Notation of balanced curves in $\M^h$ \\ 
\hline\hline
&&&\\
$1$ & $1$ & $\M_{1,(\ell_1)_1}$ & $\M_{\ell_1}^1  $  \\
$2$ & $2$ & $\M_{1,4_1+(\ell_2)_2} $ & $\M_{\ell_2}^2  $  \\
$3$ & $4$ & $\M_{1,2_1+1_2+(\ell_4)_4} $ & $\M_{\ell_4}^4  $  \\
$4$ & $6$ & $\M_{1,1_1+1_2+1_3+(\ell_6)_6} $ & $\M_{\ell_6}^6  $ \\
$5$ & $3$ & $\M_{1,3_1+(\ell_3)_3} $ & $\M_{\ell_3}^3  $ \\ 
\end{tabular} 
\end{table} 

\subsection{Permutation equivariant Dilaton equations for balanced curves}
\

In this Subsection, we prove versions of $g=1$ Dilaton equation (\cite{Lee01} Sections 4.4, 4.5) for ancestor invariants in permutation equivariant cases.

\begin{prop}
\

\begin{enumerate}
\item
\[\begin{array}{ll}
&\<\!\<\dfrac1{1-p_1\bar{L}^2},\cdots, \dfrac1{1-p_4\bar{L}^2}, \dfrac1{1-q_1\bar{L}}, \cdots, \dfrac1{1-q_{\ell_2}\bar{L}}, \bar{L}-1\>\!\>_{1,4_1+(\ell_2+1)_2}\\
=& (4+2\ell_2) \<\!\<\dfrac1{1-p_1\bar{L}^2},\cdots, \dfrac1{1-p_4\bar{L}^2}, \dfrac1{1-q_1\bar{L}}, \cdots, \dfrac1{1-q_{\ell_2}\bar{L}}\>\!\>_{1,4_1+(\ell_2)_2}.
\end{array}\]
\item
\[\begin{array}{ll}
&\<\!\<1,1, 1, \dfrac1{1-q_1\bar{L}}, \cdots, \dfrac1{1-q_{\ell_3}\bar{L}}, \bar{L}-1\>\!\>_{1,3_1+(\ell_3+1)_3}\\
=& (3+3\ell_3) \<\!\<1,1,1, \dfrac1{1-q_1\bar{L}}, \cdots, \dfrac1{1-q_{\ell_2}\bar{L}}\>\!\>_{1,3_1+(\ell_3)_3}.
\end{array}\]
\item
\[\begin{array}{ll}
&\<\!\<1,1, 1, \dfrac1{1-q_1\bar{L}}, \cdots, \dfrac1{1-q_{\ell_4}\bar{L}}, \bar{L}-1\>\!\>_{1,2_1+1_2+(\ell_4+1)_4}\\
=& (4+4\ell_4) \<\!\<1,1,1, \dfrac1{1-q_1\bar{L}}, \cdots, \dfrac1{1-q_{\ell_2}\bar{L}}\>\!\>_{1,2_1+1_2+(\ell_4)_4}.
\end{array}\]
\item
\[\begin{array}{ll}
&\<\!\<1,1, 1, \dfrac1{1-q_1\bar{L}}, \cdots, \dfrac1{1-q_{\ell_4}\bar{L}}, \bar{L}-1\>\!\>_{1,1_1+1_2+1_3+(\ell_6+1)_6}\\
=& (6+6\ell_6) \<\!\<1,1,1, \dfrac1{1-q_1\bar{L}}, \cdots, \dfrac1{1-q_{\ell_2}\bar{L}}\>\!\>_{1,1_1+1_2+1_3+(\ell_6)_6}.
\end{array}\]
\end{enumerate}
\end{prop}

We say that two $h$-bundles $V,V'$ over an $h$-orbifold $\M$ are {\em $h$-equivalent} (notation: $V\stackrel{h}{=} V'$), if on each strata (labeled by $\tilde{h}$) in $I^h\M$, we have $str_{\tilde{h}} V = str_{\tilde{h}} V'$ restricted as elements in $K^0(\M^{\tilde{h}})$.
Note that in particular, if $V=V'$ as elements in $K^0(\M^{h})$, then $V\stackrel{h}{=} V'$.

Let $V_1,V_2$ be $h$-orbi-bundles over $h$-orbifolds $\M_1,\M_2$.
We say $V_1$ {\em $h$-push-forwards} to $V_2$ along $ft: \M_1 \to \M_2$, if there are bundles $V'_1,V'_2$ $h$-equivalent to $V_1,V_2$, with $ft_\ast V'_1 = V'_2$.

We have the following proposition from \cite{perm9} (See also Section 3 of this paper).

\begin{prop}
Let $h$ be a generator of a cyclic group $H$, $\mathcal{M}$ an $H$-orbifold, and $V \to \mathcal{M}$ an orbi-bundle.
The super-trace $str_h H^\ast(\mathcal{M}, V)$ equals the following quantities.
\begin{enumerate}
\item
Set $E_h = \sum_\lambda \lambda^{-1} \mathbb{C}_{\lambda}$, where $\mathbb{C}_\lambda$ is the irreducible representation of $H$ on which $h$ acts as $\lambda$.
The holomorphic Euler characteristic
\[\chi\left(\mathcal{M}/H, (V \cdot E_h)/H\right) = \bar{p}_\ast \left((V \cdot E_h)/H\right),\]
where $\bar{p}$ is the map from $\mathcal{M}/H$ to a point.
\item
The Kawasaki-Riemann-Roch type integral
\[\chi^{fake}\left(I^h\mathcal{M}, str_{\tilde{h}} \frac{V}{\wedge^\ast \tilde{N}^\ast_{I^h\mathcal{M}/\mathcal{M}}}\right).\]
\end{enumerate}
\end{prop}

\begin{cor}
\

\begin{enumerate}
\item
$h$-equivalent bundles $V,V'$ define the same invariants
\[str_h H^\ast(\mathcal{M},V) = str_h H^\ast(\mathcal{M},V').\]
\item
If $V_1$ $h$-push-forwards to $V_2$ along $ft: \mathcal{M}_1 \to \mathcal{M}_2$, then
\[str_h H^\ast(\mathcal{M}_1,V_1) = str_h H^\ast(\mathcal{M}_2,V_2).\]
\end{enumerate}
\end{cor}

\begin{proof}
The first part directly follows from the second part of the above Proposition.
The second part follows from the first part of the corollary, the first part of the proposition, and the fact that the forgetful map commutes with the quotient by $H$, as their fibers over a point are the same.
\end{proof}

With these preparations, we prove part $1$ of Proposition 4.1.
Our proof modifies that in Sections 4.4, 4.5 in \cite{Lee01}.

\begin{lemma}
Let $ft: \M_{1,4+2(\ell_2+1)} \to \M_{1,4+2\ell_2}$ be the forgetful map that forgets the last $2$ marked point (and stabilizes the result).
The generating functions
\[\dprod_{i=1}^4 \frac1{1-p_iL^2_i}\prod_{i=1}^{\ell_2} \frac1{1-q_iL_{2i+3}L_{2i+4}} \cdot (L_{2\ell_2+5}L_{2\ell_2+6} -1)\]
$h$-push-forwards along $ft$ to
\[\left((\H^\ast+\H-2)^2 + 4+ 2\ell_2 \right)\dprod_{i=1}^4 \frac1{1-p_iL^2_i}\prod_{i=1}^{\ell_2} \frac1{1-q_iL_{2i+3}L_{2i+4}}.\]
Here $h$ permutes $L_i$'s and terms in $(\H^\ast+\H-2)^2$.
\end{lemma}

\begin{proof}
For clearness, we use $L_i$ for universal cotangent bundles on $\M_{1,4+2(\ell_2+1)}$, $\tilde{L}_i$ for that on $\M_{1,4+2\ell_2+1}$ and $L'_i$ for that on $\M_{1,4+2\ell_2}$.
We prove
\begin{enumerate}
\item
An $h$-push-forward of
\[\left(\dprod_{i=1}^4 \frac1{1-p_iL^2_i}\prod_{i=1}^{\ell_2} \frac1{1-q_iL_{2i+3}L_{2i+4}}\right) \]
along $ft$ is
\[\left((1-\H^\ast)^2 + \dsum_{i=1}^4 \frac{p_i}{1-p_i}+\sum_{i=1}^{\ell_2} \frac{2q_i}{1-q_i} \right)\dprod_{i=1}^4 \frac1{1-p_i(L'_i)^2}\prod_{i=1}^{\ell_2} \frac1{1-q_iL'_{2i+3}L'_{2i+4}},\]
with $h$ permuting $L_i, L'_i$'s and cyclically permuting terms in $(1-\H^\ast)^2$.
\item
An $h$-push-forward of
\[\left(\bigg(\dprod_{i=1}^4 \frac1{1-p_iL^2_i}\prod_{i=1}^{\ell_2} \frac1{1-q_iL_{2i+3}L_{2i+4}}\bigg) \cdot L_{2\ell_2+5}L_{2\ell_2+6}\right) \]
along $ft$ is
\[ \left((\H-1)^2 + \dsum_{i=1}^4 \frac1{1-p_i}+\sum_{i=1}^{\ell_2} \frac2{1-q_i} \right)\dprod_{i=1}^4 \frac1{1-p_i(L'_i)^2}\prod_{i=1}^{\ell_2} \frac1{1-q_iL'_{2i+3}L'_{2i+4}},\]
with $h$ permuting $L_i, L'_i$'s and cyclically permuting terms in $(\H-1)^2$.
\end{enumerate}

For $i \in \{1,\cdots, 4+ 2\ell_2\}$, let $D^1_i$ (resp. $D^2_i$ and $D^{12}_i$) be the divisor in $\M_{1,4+2(\ell_2+1)}$ consisting of curves with two components, one of which is of genus $0$ and exactly contains the marked points labeled by $i, 4+2\ell_2+1$ (resp. $i, 4+2\ell_2+2$ and $i, 4+2\ell_2+1, 4+2\ell_2+2$).
The map defined by point-wise tangent map $T_ift: T_iC \to T_i(ft(C))$ at $i^{th}$ marked point vanish exactly on $D_i^1 + D_i^2 + D_i^{12}$, so the cotangent map $T^\ast_i ft: T^\ast_i(ft(C)) \to T^\ast_iC$ also vanish exactly on $D_i^1 + D_i^2 + D^{12}_i$.
Thus, the universal cotangent bundle satisfy
\[L_i = ft^\ast (L'_i) \cdot \mathcal{O}(D_i^1 + D_i^2 + D^{12}_i).\]

Use $ft_k: \M_{1,4+2(\ell_2+1)} \to \M_{1,4+2\ell_2+1}$ for the map that forgets the $(4+2\ell_2+k)$-th marked point, and $\bar{ft}_k$ for their counterparts $\M_{1,4+2\ell_2+1} \to \M_{1,4+2\ell_2}$.
Use $\bar{D}_i^1$ (resp. $\bar{D}_i^2$) for the divisor on $\M_{1,4+2(\ell_2)+1} = ft_2 (\M_{1,4+2(\ell_2+1)})$ (resp. $\M_{1,4+2(\ell_2)+1} = ft_1 (\M_{1,4+2(\ell_2+1)})$) given by curves with two components, one of which is of genus $0$ and exactly contains the marked points labeled by $i, 4+2\ell_2+1$ (resp. $i, 4+2\ell_2+2$).
From now on, in this proof, we use $j$ for a number in $\{1,2,3,4\}$ and $i$ for a number in $\{1,2,\cdots,\ell_2\}$.

\medskip

For the first part, we have for $d_i \geq 1$ that
\[\begin{array}{ll}
\dprod_j L_j^{-2d_j} \dprod_i L_i^{-d_i} =& ft^\ast (\dprod_j (L'_j)^{-2d_j} \dprod_i (L'_i)^{-d_i}) \\
&\cdot \mathcal{O}(-\dsum_i d_i(D_i^1+D_i^2+D_i^{12}) -2\dsum_j d_j(D_j^1+D_j^2+D_j^{12})).
\end{array}\]

The divisors $D_i^{12}$ does not meet the $h$-fixed locus, so
\[\mathcal{O}(D_i^1+D_i^2+D_i^{12}) \stackrel{h}{=} \mathcal{O}(D_i^1+D_i^2+2D_i^{12}).\]
On the other hand, the section $T_i^\ast ft_2: ft^\ast L'_i \to \tilde{L}_i$ vanishes exactly on $\bar{D}_i^2$, and its pull-back along $ft_1$ vanishes exactly on $D_i^2+D_i^{12}$, so
\[\mathcal{O}(D_i^1+D_i^{12}) = ft_1^\ast \mathcal{O}(\bar{D}_i^2).\]
Likewise, we have
\[\mathcal{O}(D_i^2+D_i^{12}) = ft_2^\ast \mathcal{O}(\bar{D}_i^1).\]
Thus
\[\mathcal{O}(D_i^1+D_i^2+D_i^{12}) \stackrel{h}{=} ft_1^\ast \mathcal{O}(\bar{D}_i^2) ft_2^\ast \mathcal{O}(\bar{D}_i^1).\]

It follows similarly from the fact $D_j^1,D_j^2$ does not meet the $h$-fixed locus that
\[\mathcal{O}(2(D_j^1+D_j^2+D_j^{12})) \stackrel{h}{=} ft_1^\ast \mathcal{O}(\bar{D}_j^2) ft_2^\ast \mathcal{O}(\bar{D}_j^1).\]

As a result, 
\[\begin{array}{ll}
&\mathcal{O}(-\dsum_i d_i(D_i^1+D_i^2+D_i^{12}) -2\dsum_j d_j(D_j^1+D_j^2+D_j^{12})) \\
\stackrel{h}{=}&ft_2^\ast \mathcal{O}(-\dsum_i d_i \bar{D}_i^1 - \dsum_j d_j \bar{D}_j^1) \cdot ft_1^\ast \mathcal{O}(-\dsum_i d_i \bar{D}_i^2 - \dsum_j d_j \bar{D}_j^2).
\end{array}\]

Observe that $(ft_1)_\ast ft_2^\ast V = \bar{ft}_2^\ast (\bar{ft}_1)_\ast V$ for all bundles $V$ over $\M_{1,4+2\ell_2+1}$.
In particular,
\[(ft_1)_\ast ft_2^\ast \mathcal{O}(-\sum_i d_i \bar{D}_i^1 - \sum_j d_j \bar{D}_j^1) = \bar{ft}_2^\ast (\bar{ft}_1)_\ast \mathcal{O}(-\sum_i d_i \bar{D}_i^1- \sum_j d_j \bar{D}_j^1).\]

Together with the projection formula, we simplify the push-forward
\[(ft_2)_\ast (ft_1)_\ast \left[ft_2^\ast \mathcal{O}(-\sum_i d_i \bar{D}_i^1 - \sum_j d_j \bar{D}_j^1) \cdot ft_1^\ast \mathcal{O}(-\sum_i d_i \bar{D}_i^2 - \sum_j d_j \bar{D}_j^2)\right]\]
as
\[(\bar{ft}_2)_\ast \mathcal{O}(-\sum_i d_i \bar{D}_i^2- \sum_j d_j \bar{D}_j^2) \cdot (\bar{ft}_1)_\ast \mathcal{O}(-\sum_i d_i \bar{D}_i^1- \sum_j d_j \bar{D}_j^1).\]
On this virtual bundle $h$ acts as interchanging $\bar{D}_i^1, \bar{D}_{h(i)}^2$ and interchanging $\bar{D}_j^1, \bar{D}_j^2$.
By \cite{Lee01} Section 4.4, we know that the push-forward
\[(\bar{ft}_k)_\ast \mathcal{O} (-\sum_i d_i \bar{D}_i^k - \sum_j d_j \bar{D}_j^k) = 1-\H^\ast - \sum_i\sum_{n_i=0}^{d_i-1} (L'_i)^{-n_i}- \sum_j\sum_{n_j=0}^{d_j-1} (L'_j)^{-n_j},\]
thus the invariants 
\[ft_\ast \left(\prod_j L_j^{-2d_j} \prod_i L_i^{-d_i}\right) = \left(\prod_j (L'_j)^{-2d_j} \prod_i (L'_i)^{-d_i}\right) \cdot (1-\H^\ast - \sum_i\sum_{n_i=0}^{d_i-1} (L'_i)^{-d_i} - \sum_j\sum_{n_j=0}^{d_j-1} (L'_j)^{-n_j})^2.\]
Here $h$ takes $L'_j, L'_i$ in $\prod_j (L'_j)^{-2d_j} \prod_i (L'_i)^{-d_i}$ to $L'_j, L'_{h(i)}$, and interchanges $L'_j, L'_i$ in the first parenthesis with $L'_j, L'_{h(i)}$ in the second parenthesis in $(1-\H^\ast - \sum_i\sum_{n_i=0}^{d_i-1} (L'_i)^{-d_i} - \sum_j\sum_{n_j=0}^{d_j-1} (L'_j)^{-n_j})^2$.
Put this into a generating function and take the $h$-invariant part, we obtain the first part.
Here we remark that if $h$ interchanges the $(2i+3)$-rd and $(2i+4)$-th marked points, then
\[str_h \frac1{1-q_i L_{2i+3}}\frac1{1-q_i L_{2i+4}} = \frac1{1-q_i L_{2i+3}L_{2i+4}}.\]
(and is not $\frac1{1+q^2_i L_{2i+3}L_{2i+4}}$).

To prove the second point, we need to ($h$-)push-forward
\[ft^\ast (\prod_j (L'_j)^{2d_j}\prod_i (L'_i)^{d_i}) \cdot \left(L_{4+2\ell_2+1}L_{4+2\ell_2+2} \right)(\sum_i d_i(D_i^1+D_i^2+D_i^{12}) + 2\sum_j d_j(D_j^1+D_j^2+D_j^{12}))\]
along $ft$.
Observe that the universal cotangent bundle $L_{4+2\ell_2+1}$ and the pull-back of its counterpart $\tilde{L}_{4+2\ell_2+1}$ along $ft_2$ coincide outside the divisor formed by curves that contains a $g=0$ nodal component with precisely the last two marked points, and that no point in this divisor is fixed by $h$.
Thus $L_{4+2\ell_2+1} \stackrel{h}{=} ft_2^\ast \tilde{L}_{4+2\ell_2+1}$.
Identical to part $1$, this implies that 
\[\begin{array}{ll}
&L_{4+2\ell_2+1}L_{4+2\ell_2+2} \cdot \mathcal{O}(\dsum_i d_i(D_i^1+D_i^2+D_i^{12}) + 2\dsum_j d_j(D_j^1+D_j^2+D_j^{12})) \\
\stackrel{h}{=} & (ft_2^\ast\tilde{L}_{4+2\ell_2+1})\mathcal{O}(\dsum_i d_i\bar{D}_i^1+ \dsum_j d_j\bar{D}_j^1) \cdot (ft_1^\ast L_{4+2\ell_2+2})\mathcal{O}(\dsum_i d_i\bar{D}_i^2+ \dsum_j d_j\bar{D}_j^2),
\end{array}\]
which pushes forward along $ft$ to
\[(ft_2)_\ast \tilde{L}_{4+2\ell_2+2}(\sum_i d_i \bar{D}_i^2+ \sum_j d_j\bar{D}_j^2) \cdot (ft_1)_\ast \tilde{L}_{4+2\ell_2+1}(\sum_i d_i \bar{D}_i^1 + \sum_j d_j \bar{D}_j^1).\]
It follows from the fact
\[\tilde{L}_{4+2\ell_2+k} = \omega(\sum_i \bar{D}_i^k+\sum_j \bar{D}_j^k)\]
and Section 4.5 in \cite{Lee01} that the desired $h$-push-forward is
\[(\prod_j (L'_j)^{2d_j}\prod_i (L'_i)^{d_i}) \cdot (\H-1+\sum_i\sum_{n_i=0}^{d_i} (L'_i)^{-n_i}+\sum_j\sum_{n_j=0}^{d_j} (L'_j)^{-n_j})^2.\]
Here, the $h$ action is identical to that in the first part, and the second part follows likewise.
\end{proof}

By pulling back the lemma along forgetful maps $\M_{1,n+\bar{n}} \to \M_{1,n}$ and multiplying $\tau$-contributions (that does not contain universal cotangent bundles) from the last $\bar{n}$ marked points, we have

\begin{cor}
\[\dprod_{i=1}^4 \frac1{1-p_i\bar{L}_i}\prod_{i=1}^{\ell_2} \frac1{1-q_i\bar{L}_{2i+3}\bar{L}_{2i+4}} \cdot (\bar{L}_{2\ell_2+5}\bar{L}_{2\ell_2+6} -1) \cdot \prod_{i=1}^{\bar{n}} \tau_{r_i}\]
$h$-push-forwards along $ft$ that forgets the last two (out of $n$) points to
\[\left((\H^\ast+\H-2)^2 + 4+ 2\ell_2 \right)^2\dprod_{i=1}^4 \frac1{1-p_i\bar{L}_i}\prod_{i=1}^{\ell_2} \frac1{1-q_i\bar{L}_{2i+3}\bar{L}_{2i+4}} \cdot \prod_{i=1}^{\bar{n}} \tau_{r_i}.\]
\end{cor}

When $M=3,4,6$, let $h$ cyclically permutes the last $M$ marked points in $\M^M_{\ell_M+1}$, and $ft:\M^M_{\ell_M+1} \to \M^M_{\ell_M}$ forgets these points. 
Let $\mathcal{H}$ be the Hodge bundle over $\M^M_{\ell_M}$.
The proof of the following lemma is identical to that of Corollary 4.2.

\begin{lemma}
An $h$-push-forward of 
\[\left(\prod_{i=1}^{\ell_M} \frac1{1-q_i\bar{L}_{Mi+M+1} \cdots \bar{L}_{Mi+2M}} \cdot (\bar{L}_{M\ell_M+M+1} \cdots \bar{L}_{M\ell_M+2M}-1) \cdot \prod_{i=1}^{\bar{n}} \tau_{r_i}\right)\]
along $ft$ is
\[\left((\H^\ast+\H-2)^M + M + M\ell_M \right) \prod_{i=1}^{\ell_M} \frac1{1-q_i\bar{L}_{Mi+M+1} \cdots \bar{L}_{Mi+2M}} \cdot \prod_{i=1}^{\bar{n}} \tau_{r_i}.\]
with $h$ permuting $L_i$'s and cyclically permuting terms in $(\H^\ast+\H-2)^M$.
\end{lemma}

\begin{lemma}
Substituting $\mathcal{H}$ by $1$ preserves the correlators.
\end{lemma}

\begin{proof}
Note that on a strata labeled by $\tilde{h} = (h,i)$, where $i$ is a symmetry of $\M_{1,n}^h$ and $h$ is of order $M \geq 2$, $\tilde{h}$ acts on $(\mathcal{H}^\ast + \mathcal{H} - 2)^M$ as permuting components by $h$ and as $i$ (trivially) on each component, thus $\tilde{h}$ acts as $h$ on $str_h(\mathcal{H}^\ast + \mathcal{H} - 2)^M = (\mathcal{H}^\ast)^M + \mathcal{H}^M - 2$.

The Hodge bundle $\mathcal{H}$ satisfies $\mathcal{H}^\ast = \mathcal{H}^{-1}$ and thus $ch \ str_h(\mathcal{H}^\ast + \mathcal{H} - 2)^M$ is a multiple of $ch (\mathcal{H}-1)^2/\mathcal{H}$.
So $ch \ str_h(\mathcal{H}^\ast + \mathcal{H} - 2)^M$ has degree at least $4$.
For dimension reasons, terms with $\H^\ast + \H - 2$ vanish.
\end{proof}

Now, taking the weighted sum of Corollary 4.2, Lemma 4.3, and Lemma 4.4 into account, we have Proposition 4.1.

\begin{lemma}
We have the following super-traces.
\begin{enumerate}
\item
\[str_h H^\ast\left(\M_{1,4+2\ell_2}^{-1}, (-L_1-1) \cdot \prod_{i=1}^{\ell_2}(L_{2i+3}L_{2i+4}-1)\right) = \frac14 \cdot 2^{\ell_2} (\ell_2+1)!\]
\item
\[str_h H^\ast\left(\M_{1,4+2\ell_2}^{-1}, (L_5L_6-1)^2 \cdot \prod_{i=2}^{\ell_2}(L_{2i+3}L_{2i+4}-1)\right) = 2^{\ell_2-2} (\ell_2+1)!\]
\end{enumerate}
\end{lemma}

\begin{proof}
For the first point, for dimension reasons 
\[str_h H^\ast\left(\M_{1,4+2\ell_2}^{-1}, \frac{(-L_1-1)^2}2 \cdot \prod_{i=1}^{\ell_2}(L_{2i+3}L_{2i+4}-1)\right) = 0,\]
and adding this to the left hand side changes the contribution of the first marked point to the sheaf to $(L_1^2-1)/2$, making our version of dilaton equation Proposition 4.1 applicable.
Our Dilaton equation eliminates all but the first four marked points, and simplifies the supertrace to 
\[2^{\ell_2} (\ell_2+1)! \cdot str_h H^\ast\left(\M_{1,4}^{-1}, \frac{(L_1-1)^2}2\right) = 2^{\ell_2} (\ell_2+1)! \cdot str_h H^\ast\left(\M_0^2, (-L_1-1)\right).\]
The forgetful map $\M^2_0 \to \M_{1,1}$ that forgets all but the first marked point is a $3!=6$-fold covering.
Thus the correlators
\[str_h H^\ast(\M_0^2,L_1-1) = \int_{\M_0^2} ch(L_1-1) = 6\int_{\M_{1,1}} c_1(L_1)=\frac14.\]

For the second point, the Dilaton equation eliminates all but the first six marked points, and contributes the coefficient $2^{\ell_2-2} (\ell_2+1)!$.
The remaining supertrace is
\[str_h H^\ast\left(\M_{1,6}^{-1}, (L_5L_6-1)^2\right) = \frac12 str_h H^\ast\left(\M_1^2, (L_5L_6-1)^2\right),\]
after noticing that $\M_1^2$ is the fixed locus of the $\mathbb{Z}_2$ symmetry of $\M_{1,6}^{-1}$, with co-dimension $1$.
Next, the forgetful map $\M_1^2 \to \M_{1,2}$ that forgets the first $4$ level $1$ marked points is a $24$-fold covering map, where $24$ comes from $4!=24$ ways to label the four level $1$ marked points.
Moreover, the identification given by forgetting the first four marked points preserves the universal cotangent bundles at the other $2$ marked points and the $\mathbb{Z}_2$ action on them.
We have
\[\begin{array}{ll}
&str_{h_2} H^\ast(\M^2_1,(L_5L_6-1)^2)\\
=&\dint_{[\M_1^2]} ch \ str_{h_2} \dfrac{(L_5L_6-1)^2}{\wedge^\ast N^\ast_{\M_0^2/\M_1^2}} \cdot td (\M_1^2)\\
=&\dint_{[\M_1^2]} ch \dfrac{(L_5L_6-1)^2}2\\
=&12 \dint_{[\M_{1,2}]} ch (L_1L_2-1)^2.
\end{array}\]
Since
\[L_1L_2-1=(L_1-1)(L_2-1)+(L_1-1)+(L_2-1),\]
by degree argument and the cohomological String and Dilaton equations, we have
\[12 \dint_{[\M_{1,2}]} ch (L_1L_2-1)^2 = 12 \dint_{[\M_{1,2}]} (c_1(L_1)+c_1(L_2))^2=2.\]
\end{proof}

Applying the (virtual) Kawasaki's Riemann Roch theorem \cite{Tonita14} yields
\begin{cor}
\[\<\!\<-\bar{L}-1,1,1,1;\bar{L}-1, \cdots \>\!\>_{1,4_1+\ell_2} = \<\!\<1,1,1,1;(\bar{L}-1)^2, \bar{L}-1, \cdots \>\!\>_{1,4_1+\ell_2}\]
\end{cor}

\subsection{Proof of Theorem 3}
\

\begin{lemma}
Let $\bar{t}_{r,1}$ be the leading coefficients of $\bar{\t}_r(q) = \bar{t}_{r,1}(q-1) + O((q-1)^2)$.
Then
\[\frac1{1-\bar{t}_{r,1}} = \frac{\partial \tau_r}{\partial t_{r,0}}.\]
\end{lemma}

\begin{proof}
The proof is identical to that of Proposition 1.1 in \cite{nonperm} in the special case of point target space after noticing that 
\[G_r^{-1} R_r \<\!\<1,1-\bar{t}_{r,1},1\>\!\>_{0,3_r} = 1-\bar{t}_{r,1}.\]
\end{proof}

\

We now compute the term
\[\bar{\mathcal{F}}_1(\bar{\t})= \dsum_{\ell,\bar{\ell}} \frac1{\ell!\bar{\ell}!} \left<\bar{\mathbf{t}}_1(\bar{L}), \cdots, \bar{\mathbf{t}}_1(\bar{L}), \tau_1,\cdots,\tau_1;\cdots, \bar{\mathbf{t}}_r(\bar{L}), \cdots, \tau_r,\cdots \right>_{1,\ell+\bar{\ell}} \]
in Ancestor-Descendant correspondence.
By definition, each correlator $\<\cdots\>$ is a supertrace of $h$ on some sheaf cohomology.

When $ord(h)=1$ on $\M_{1,\ell}$ (i.e., with the last $\bar{\ell}$ marked points carrying inputs $\tau$'s forgotten), identical to the non-permutative case \cite{nonperm}, the contribution of the non-twisted stratum is
\[\frac1{24} \log \det\left(\frac{\partial \tau_1}{\partial t_{1,0}}\right).\]
Contributions from twisted strata are in $\ell_M = 0$, $M \geq 2$ cases.

\

When $ord(h)=2$ on $\M_{1,\ell}$, expand
\[\begin{array}{llll} 
\bar{\t}_2(q) &= \bar{t}_{2,1} (q-1) &+ \bar{t}'_{2,1} (q-1)^2 &+ O((q-1)^3); \\
\bar{\x}_1(-q)&= \bar{x}_{1,-1} &+ \bar{x}'_{1,-1} (q-1) &+ O((q-1)^2).
\end{array}\]
Note that $\bar{t}_{2,1}, \bar{t}'_{2,1}$ are $\bar{t}_{2,1}, \bar{t}_{2,2}$ in the introduction.

When $\ell_2=0$, the lift of $h$ with respect to the twisted strata and non-twisted strata acts on $\bar{L}$'s in the same way.
The dimension of $\M_{1,4+2\ell_2}$ is $\ell_2 +1$, so only invariants with one of the following $3$ types of inputs contribute.

\begin{table}[H]
\centering       
\begin{tabular}{ c | l}                       
Case & Inputs \\
\hline\hline
&\\
$a$  & All points $\bar{x}_{1,-1}, \bar{t}_{2,1}(q-1)$ \\
$b$  & In general $\bar{x}_{1,-1}, \bar{t}_{2,1}(q-1)$, one point $\bar{x}'_{1,-1}(-q-1)$ \\
$c$  & In general $\bar{x}_{1,-1}, \bar{t}_{2,1}(q-1)$, one pair $\bar{t}'_{2,1}(q-1)^2$
\end{tabular}
\end{table} 

In case $a$, the dilaton equation eliminates all marked points with input $\bar{t}_{2,1}(\bar{L}-1)$, and simplifies the correlators as
\[\<\!\<1,1,1,1\>\!\>_{1,4_1} \bar{x}_{1,-1}^4 \cdot \dfrac1{(1-\bar{t}_{2,1})^2}.\]

In cases $b$, by Corollary 4.3 and the Dilaton equation Proposition 4.1, we simplify the invariants as
\[\<\!\<-\bar{L}-1,1,1,1\>\!\>_{1,4_1} \bar{x}_{1,-1}^3 \bar{x}'_{1,-1} \cdot \dfrac1{(1-\bar{t}_{2,1})^2}.\]

In case $c$, Corollary 4.3 and Proposition 4.1 simplifies the invariants as
\[\<\!\<1,1,1,1,(\bar{L}-1)^2\>\!\>_{1,4_1+1_2} \bar{x}_{1,-1}^4 \cdot \dfrac1{(1-\bar{t}_{2,1})^3} \bar{t}'_{2,1} = \<\!\<-\bar{L}-1,1,1,1,\>\!\>_{1,4_1} \bar{x}_{1,-1}^4 \cdot \dfrac1{(1-\bar{t}_{2,1})^3} \bar{t}'_{2,1}.\]

\

When $ord(h)=3,4,6$ on $\M_{1,\ell}$, the invariants can be likewise written as the last $3$ terms in the residue, by Proposition 4.1.

\section*{Appendix. Lagrangian cone formalism}
\

In this appendix, we consider the range $\mathcal{L}$ of the big $\mathcal{J}$-function 
\[\mathcal{J} (\t_1) = 1-q+\t_1 + \sum_{\ell, d,\alpha} \frac{Q^d}{\ell !} \phi^\alpha \<\frac{\phi_\alpha}{1-qL}, \t_1,\cdots,\t_1;\cdots,\t_r\cdots\>_{0,1_1+\ell,d},\]
viewed as a function in $\t_1$, with parameters $\t_2,\cdots,\t_r \cdots$.

We aim at proving
\begin{customthm}{4}
$\mathcal{L} = \cup_\tau (1-q) \S_{1,\tau}^{-1} \mathcal{K}_+$ is an overruled Lagrangian cone.
\end{customthm}

\begin{lemma*}
$\mathcal{L}$ is an overruled variety.
\end{lemma*}

\begin{proof}
The Adelic characterization describes the big $\mathcal{J}$ function $\mathcal{J}(\t_1)$ as follows.
For each root of unity $\zeta$, let $\mathcal{J}^{(\zeta)}$ be the Laurent series expansion of $\mathcal{J}(q^{1/m}/\zeta)$ at $q=1$.
Let 
\[\Delta_\zeta = \exp \left[\sum_{k\geq 1} \left(\frac{\Psi^k(T^\ast X-1)}{k(1-\zeta^{-k}q^{k/m})} - \frac{\Psi^{km}(T^\ast_X -1)}{k(1-q^{km})}\right) \right].\]
There are points $f_m$'s for $m \geq 2$ determined by $\t_2,\cdots, \t_r,\cdots$ on a certain overruled (i.e., each of its tangent space $\mathcal{T}$ meets $\mathcal{L}$ at $(1-q)\mathcal{T}$) Lagrangian cone $\mathcal{L}^{fake}$, such that the range of the big $\mathcal{J}$ function is characterized by the following conditions.
\begin{enumerate}
\item
$\mathcal{J}^{(1)} \in \mathcal{L}^{fake}$.
\item
For each $\zeta \neq 1$, $\mathcal{J}^{(\zeta)}$ lies in the Lagrangian subspace
\[T_\zeta = \Delta_\zeta \Psi^m(T_{f_m}\mathcal{L}^{fake}) \otimes_{\Psi^m(\Lambda)}\Lambda.\]
\end{enumerate}

For a given $\t_1 \in \mathcal{K}_+$, let $T+\mathcal{J}(\t_1)$ (resp. $T_1$) be the tangent space of $\mathcal{L}$ (resp. $\mathcal{L}^{fake}$) at $\mathcal{J}(\t_1)$ (resp. $\mathcal{J}^{(1)}(\t_1)$).
Following from the fact that $\mathcal{L}^{fake}$ is an overruled cone, we know that 
\[(1-q)T_1 \subseteq \mathcal{L}^{fake}; \ \ \ (1-q)T_\zeta \subseteq T_\zeta.\]

Thus points in $(1-q)T$, and hence $\mathcal{J}(\t_1)+(1-q)T$, pass tests (1)(2) in Adelic characterization, and so $\mathcal{J}(\t_1)+(1-q)T \subseteq \mathcal{L}$ for any $\t_1 \in \mathcal{K}_+$.
\end{proof}

Let $\tilde{\mathcal{F}}_0 = \mathcal{F}_0 - \<\!\<\t_2\>\!\>_{0,1_2}$
Let 
\[\begin{array}{ll}
\tilde{\mathcal{L}} &= \{1-q+\t_1 + d_{\t_1} \tilde{\mathcal{F}}_0\};\\
\bar{\mathcal{L}} &= \{1-q+\bar{\t}_1 + d_{\bar{\t}_1} \bar{\mathcal{F}}_0\}.
\end{array}\]

Also recall that 
\[\begin{array}{ll}
\S_1(\phi)&=\dsum_{\alpha,\beta} \left((\phi,\phi_\alpha) + \<\!\<\dfrac{\phi}{1-L/q},\phi_\alpha\>\!\>_{0,2_1}\right)G^{\alpha\beta}_1\phi_\beta, \\
\S_1^{-1}(\phi) &= \phi + \dsum_{\alpha,\beta} \<\!\<\phi,\dfrac{\phi_\alpha}{1-qL}\>\!\>_{0,2_1} g^{\alpha\beta}\phi_\beta,
\end{array}\]
on $K$ and extended $\mathbb{Q}((q))$-linearly to $\mathcal{K}$.
In particular, $\S_1$ maps $(\mathcal{K}_1)_-$ to $(\bar{\mathcal{K}}_1)_-$, and $\S^{-1}_1$ maps $(\bar{\mathcal{K}}_1)_-$ to $(\mathcal{K}_1)_-$.

\begin{lemma*}
$\tilde{\mathcal{L}} = \S_1^{-1} \bar{\mathcal{L}}$.
\end{lemma*}

\begin{proof}
We show that when $\bar{\t}_1= [\S_1\t_1]_+$,
\[1-q+\t_1 + \tau_1 + d_{\t_1} \tilde{\mathcal{F}}_0(\t_1+\tau_1) = \S_1^{-1}(1-q+\bar{\t}_1 + d_{\bar{\t}_1} \bar{\mathcal{F}}_0(\bar{\t}_1)).\]

We decompose $\tilde{\mathcal{F}}_0$ into two parts, the $unstable$ part (Lemma 2.5)
\[\tilde{\mathcal{F}}^u_0(\t_1,\tau_1) = \frac12\<\!\<\mathbf{t}_1-L+1+\tau_1,\mathbf{t}_1-L+1+\tau_1\>\!\>_{0,2_1}\]
and the $stable$ part 
\[\tilde{\mathcal{F}}^s_0(\t_1,\tau_1) = \tilde{\mathcal{F}}_0(\t_1+\tau_1) - \tilde{\mathcal{F}}^u_0(\t_1,\tau_1).\]

Thus, the left and right-hand sides decompose into $3$ parts, and we identify these parts.
\[\begin{array}{llll}
LHS &= \left[1-q+\t_1 + \tau_1\right] &+ \left[d_{\t_1} \tilde{\mathcal{F}}^u_0(\t_1,\tau_1)\right] &+ \left[d_{\t_1} \tilde{\mathcal{F}}^s_0(\t_1,\tau_1)\right]; \\
RHS & = \left[[\S_1^{-1}(1-q+\bar{\t}_1)]_+\right] & + \left[[\S_1^{-1}(1-q+\bar{\t}_1)]_-\right] & + \left[d_{\bar{\t}_1} \bar{\mathcal{F}}_0(\bar{\t}_1))\right].
\end{array}\]

\begin{enumerate}
\item
Since $\S_1, \S_1^{-1}$ maps $\mathcal{K}_-$ to $\mathcal{K}_-$, we have $[\S_1^{-1}[\S_1 \bullet ]_+]_+ = \bullet$.
It follows from Lemma 2.4 and $\bar{\t}_1= [\S_1\t_1]_+$ that 
\[[\S_1 (1-q+\t_1 + \tau_1)]_+ = 1-q+\bar{\t}_1.\]
Applying $[\S_1^{-1} \bullet]_+$ to both sides identifies the first term.

\item
We identify the second and third terms by identifying their pairing with arbitrary $v \in \mathcal{K}_+$.

For the second term, use $u(q)$ for $\mathbf{t}_1(q)-q+1+\tau_1$.
Then $1-q+\bar{\t}_1=[\S_1 u]_+$, and the pairings $\Omega_1(\bullet,v)$ are respectively $\<\!\<u(L),v(L)\>\!\>_{0,2_1}$ and $\Omega_1(\S_1^{-1}[\S_1 u]_+, v) = \bar{\Omega}_1([\S_1 u]_+,\S_1 v)$.
By Cauchy's residue formula
\[[f(q)]_+ = -(Res_{w=0}+Res_{w=\infty}) \frac{f(w)dw}{w-q},\]
we have 
\[\<\!\<u,v\>\!\>_{0,2_1} = Res_{x=0,\infty}Res_{y=0,\infty} \<\!\<\frac{u(x)}{1-L/x}, \frac{v(y)}{1-L/y}\>\!\>_{0,2_1} \frac{dx}x \frac{dy}y.\]
Here we shorthand $Res_{q=0}+Res_{q=\infty}$ as $Res_{q=0,\infty}$.
By the WDVV equation, the double-bracket correlator is
\[\frac1{1-x^{-1}y^{-1}} \left(-(u(x),v(y)) + G(\S_1(x)u(x),\S_1(y)v(y))\right),\]
where $G$ is the metric $G_{1,\alpha\beta} \phi^\alpha \otimes \phi^\beta$.
By using Cauchy's residue formula, we simplify
\[Res_{x=0,\infty}Res_{y=0,\infty}\frac1{1-x^{-1}y^{-1}} G(\S_1(x)u(x),\S_1(y)v(y)) \frac{dx}x \frac{dy}y\] 
as $\bar{\Omega}_1([\S_1 u]_+, \S_1 v)$.
Likewise 
\[Res_{x=0,\infty}Res_{y=0,\infty}\frac1{1-x^{-1}y^{-1}} (u(x),v(y)) \frac{dx}x \frac{dy}y = \Omega_1([u]_+,v) = 0,\]
and thus identifying the second term.

\item
Use $\mathfrak{L}$ for Lie derivative, and by the fact that $\S_1$ is symplectic, we have the pairings $\Omega_1(\bullet,v)$ of the third terms as $\mathfrak{L}_v(\tilde{\mathcal{F}}^s_0)(\t_1)$ and $\Omega_1(d_{\bar{\t}_1}\bar{\mathcal{F}}^s_0,[\S_1v]_+) = \mathfrak{L}_{[\S_1 v]_+}(\bar{\mathcal{F}}_0)(\bar{\t}_1)$.
Note that the stable part $\tilde{\mathcal{F}}^s_0$ is the contribution from curves on which there are at least $3$ marked points with input $\t$'s, so $\mathfrak{L}_v(\tilde{\mathcal{F}}^s_0)(\t_1)= \mathfrak{L}_{[\S_1 v]_+}(\bar{\mathcal{F}}_0)(\bar{\t}_1)$ by Proposition 2.1.
\end{enumerate}
\end{proof}

Note that in the lemma, $\t$ and $\tau$ are not related, so $\<\!\<\t_2\>\!\>_{0,1_2}$ does not depend on $\t_1$.
Thus $\mathcal{L} = \tilde{\mathcal{L}} = \S_1^{-1}\bar{\mathcal{L}}$.

\begin{lemma*}
The tangent space of $\mathcal{L}$ at $\mathcal{J}(\tau_1)$ is $\S_{1,\tau}^{-1} \mathcal{K}_+$
\end{lemma*}

\begin{proof}
By Lemma 2, the tangent space of $\mathcal{L}$ at $\mathcal{J}(\tau_1)$ is the image under $\S_1^{-1}$ of the tangent space of $\bar{\mathcal{L}}_\tau$ at $\S_1(\mathcal{J}(\tau_1))$.
By Lemma 2.4 we have $[\S_{1,\tau}(1-q+\tau_1)]_+=1-q$, so $[\S_{1,\tau}\mathcal{J}(\tau_1)]_+ = 1-q$.
The point on $\bar{\mathcal{L}}_\tau$ with $\bar{\t}_1=0$ is 
\[1-q+\bar{\t}_1(q)+\sum_{\ell,\alpha,\beta} \frac{\phi_\beta G^{\alpha\beta}_1}{\ell!} \<\!\<\frac{\phi_\alpha}{1-q\bar{L}},\bar{\t}_1(\bar{L}),\cdots,\bar{\t}_1(\bar{L});\cdots,\bar{\t}_r(\bar{L}),\cdots\>\!\>_{0,1_1+\ell} = 1-q,\]
thus $\S_{1,\tau}\mathcal{J}(\tau_1) = 1-q$.
By Lemma 1, the the tangent space of $\bar{\mathcal{L}}_\tau$ at this point consist of vectors
\[v(q)+\sum_{\ell,\alpha,\beta} \frac{\phi_\beta G^{\alpha\beta}_1}{\ell!} \<\!\<\frac{\phi_\alpha}{1-q\bar{L}}, v(\bar{L}),\bar{\t}_1(\bar{L}),\cdots,\bar{\t}_1(\bar{L});\cdots,\bar{\t}_r(\bar{L}),\cdots\>\!\>_{0,2_1+\ell} = v(q),\]
i.e., the tangent space $T_{1-q}\bar{\mathcal{L}}_\tau = \mathcal{K}_+$.
Therefore $T_{\mathcal{J}(\tau_1)}\mathcal{L}=\S_{1,\tau}^{-1}\mathcal{K}_+$.
\end{proof}

Since $\S_1$ maps $(\mathcal{K}_1)_-$ to $(\bar{\mathcal{K}}_1)_-$, and $\S^{-1}_1$ maps $(\bar{\mathcal{K}}_1)_-$ to $(\mathcal{K}_1)_-$, for any $\t_1 \in \mathcal{K}_+$, we have
\[[\S_{1,\tau}[\S_{1,\tau}^{-1}(\t_1)]_+]_+ = \t_1.\]
Thus $1-q+\t_1 \in [(1-q)\S_{1,\tau}^{-1}\mathcal{K}_+]_+$ if and only if $[\S_{1,\tau}\t_1]_+(1)=0$.
As a result, ruling spaces $(1-q)\S_{1,\tau}^{-1}\mathcal{K}_+$ project to a partition of $\mathcal{K}_+$.
Thus following Lemma 1, $(1-q)\S_{1,\tau}^{-1}\mathcal{K}_+$ parametrizes the whole of $\mathcal{L}$.
In particular, each $(1-q)\S_{1,\tau}^{-1}\mathcal{K}_+$ passes through the origin, so $\mathcal{L}$ is a cone.

Moreover, $\mathcal{K}_+$, hence $\S_{1,\tau}^{-1}\mathcal{K}_+$, is Lagrangian.
So $\mathcal{L}$ is a Lagrangian cone.

\author{Dun Tang, Department of Mathematics, University of California, Berkeley, 
1006 Evans Hall, University Drive, Berkeley CA94720.
E-mail address: dun\_tang@math.berkeley.edu}

\end{document}